\documentclass[]{siamart1116}

\usepackage{amsfonts}
\usepackage{amssymb,mathtools,amsmath}
\usepackage{fixmath}
\usepackage{graphicx}
\usepackage{braket}
\usepackage{upgreek}
\usepackage{enumerate}

\usepackage{tabularx}
\usepackage{array,booktabs} %Breite Linien in Tabellen

\newtheorem{assumption}{Assumption}[section]
\newtheorem{remark}[theorem]{Remark}

\crefname{assumption}{Assumption}{Assumptions}
\crefname{proposition}{Proposition}{Propositions}
\crefname{section}{Section}{Sections}
\crefname{subsection}{Subsection}{Subsections}

\ifpdf
\DeclareGraphicsExtensions{.eps,.pdf,.png,.jpg}
\else
\DeclareGraphicsExtensions{.eps}
\fi

\numberwithin{theorem}{section}
\numberwithin{equation}{section}

\newcommand{\TheTitle}{A priori Error Estimates for\\ Space-Time Finite Element Discretization of\\ Parabolic Time-Optimal Control Problems} 
\newcommand{\TheAuthors}{Lucas Bonifacius, Konstantin Pieper, and Boris Vexler}

\headers{Error Estimates Time-Optimal Control Problems}{\TheAuthors}

\title{{\TheTitle}\thanks{\textbf{Funding: }The first author gratefully acknowledge support from the International Research
			Training Group IGDK, funded by the German Science Foundation (DFG) and the Austrian
			Science Fund (FWF).}}

\author{
	Lucas Bonifacius\thanks{Fakult\"at f\"ur Mathematik, Technische Universit\"at M\"unchen	(\email{lucas.bonifacius@tum.de}; \email{vexler@ma.tum.de}).}
	\and
	Konstantin Pieper\thanks{Department of Scientific Computing, Florida State University (\email{kpieper@fsu.edu}).}
	\and
	Boris Vexler\footnotemark[2]
}

\ifpdf
\hypersetup{
  pdftitle={\TheTitle},
  pdfauthor={\TheAuthors}
}
\fi

\crefname{assumption}{Assumption}{Assumptions}
\crefname{proposition}{Proposition}{Propositions}
\crefname{corollary}{Corollary}{Corollaries}

\newcommand{\constraintSet}{\colon}

\newcommand{\N}{\mathbb{N}}
\newcommand{\R}{\mathbb{R}}
\newcommand{\Rplus}{\R_+}

\newcommand{\Lap}{\upDelta}
\newcommand{\Q}{\ensuremath{Q}}
\newcommand{\Qad}{\ensuremath{Q}_{ad}}
\newcommand{\Lcontrol}[2][I]{\ensuremath{L^{#2}(#1\times\omega)}}
\newcommand{\LcontrolSpatial}[1][2]{\ensuremath{L^{#1}(\omega)}}
\newcommand{\ControlOp}{\ensuremath{B}}

\newcommand{\dom}[2]{\mathcal{D}_\ensuremath{{#1}(#2)}}
\newcommand{\ProjQad}[1]{\ensuremath{P_{\Qad}\left(#1\right)}}
\newcommand{\ProjectionQad}{P_{\Qad}}

\newcommand{\CriticalCone}{C_{(\bar{\nu},\bar{q})}}

\newcommand{\embedding}{\hookrightarrow}

\newcommand{\MPRHilbert}[3]{\ensuremath{H^{1}(#1;#2)\cap L^{2}(#1;#3)}}

\newcommand{\ldef}{\coloneqq}
\newcommand{\rdef}{\eqqcolon}
\newcommand{\e}{\mathrm{e}}

\newcommand{\D}[1]{\,\mathrm{d}\ensuremath{#1}}

\newcommand{\ProjDiscControl}{\mathrm{I}_\sigma}
\newcommand{\ProjK}{\mathrm{\Pi}_{k}}
\newcommand{\ProjH}{\mathrm{\Pi}_{h}}
\newcommand{\ProjKH}{\mathrm{\Pi}_{kh}}
\newcommand{\InterpolantCellwiseLinear}{\mathrm{I}_{h}}
\newcommand{\B}{\mathrm{B}}

\newcommand{\Xkh}{\ensuremath{X_{k,h}}}
\newcommand{\Qsigma}{\ensuremath{Q}_\sigma}
\newcommand{\Qsigmaad}{\ensuremath{Q}_{ad,\sigma}}

\DeclarePairedDelimiter{\abs}{\lvert}{\rvert}
\DeclarePairedDelimiter{\norm}{\lVert}{\rVert}
\DeclarePairedDelimiter{\pair}{\langle}{\rangle}
\DeclarePairedDelimiter{\inner}{(}{)}
\DeclarePairedDelimiter{\prodnorm}{\lVert(}{)\rVert}

\begin{document}

\maketitle

% REQUIRED
\begin{abstract}
Space-time finite element discretizations of
time-optimal control problems governed by linear parabolic PDEs and subject to pointwise control
constraints are considered. Optimal a priori error estimates are obtained for 
the control variable based on a
second order sufficient optimality condition.
\end{abstract}

% REQUIRED
\begin{keywords}
  Time-optimal control, Error estimates, Galerkin method
\end{keywords}

% REQUIRED
\begin{AMS}
49K20, % Calculus of variations and optimal control; optimization > Optimality conditions > Problems involving partial differential equations
49M25, % Calculus of variations and optimal control; optimization > Discrete approximations
65M15, % Numerical analysis > Error bounds
65M60 % Numerical analysis > Finite elements, Rayleigh-Ritz and Galerkin methods, finite methods
\end{AMS}

\section{Introduction}
We consider the space-time finite element discretization of the following class of
time-optimal control problems, where $u$ denotes the
state, $q$ the control, and $T$ the terminal time:
\begin{equation}\label{P}\tag{\ensuremath{P}}
\begin{aligned}
\mbox{Minimize~} j(T, q) &\ldef T + \frac{\alpha}{2}\int_0^T \norm{q(t)}^2_{\LcontrolSpatial}\D{t} \mbox{,}\\
\text{subject to}\quad &
\left\{
\begin{aligned}
T &> 0\mbox{,}\\
\partial_t u -\Lap u &= Bq\mbox{,} &\mbox{in~} (0,T)\times\Omega\mbox{,}\\
u &= 0\mbox{,} & \mbox{on~} (0,T)\times \partial\Omega\mbox{,}\\
u(0) &= u_0\mbox{,} & \mbox{in~} \Omega\mbox{,}\\
G(u(T)) &\leq 0\mbox{,}\\
q_a &\leq q(t) \leq q_b\mbox{,} &\mbox{in~} \omega\mbox{,~} t\in (0,T)\mbox{.}
\end{aligned}
\right.
\end{aligned}
\end{equation}
Here, $B$ is the control operator, $q_a, q_b \in \R$ are the control constraints, and the
terminal constraint on the state is expressed in terms of the function $G$, which is
defined as
\begin{equation}\label{eq:terminal_constraint}
G(u) \ldef \frac{1}{2}\norm{u-u_d}_{L^2(\Omega)}^2 - \frac{\delta_0^2}{2},
\end{equation}
where the desired state $u_d \in H_0^1(\Omega)$ and $\delta_0 > 0$ are given problem data.
Moreover, $\alpha > 0$ is a fixed cost parameter.
The precise assumptions will be given in \cref{sec:notation_assumptions}.
Thus, the goal is to steer the heat-equation from an initial state $u_0$ into a ball of radius
\(\delta_0\) around $u_d$, while minimizing the length of the
control horizon plus a quadratic cost term for the control.

Time-optimal control of partial differential equations is of general interest: in many
applications, a certain optimization criterion has to be met after some time, which should
be chosen as short as possible.
This does not only include the classical case where one is plainly interested to find an
admissible control that reaches the target set in minimal time, but also problems where additional
cost or regularization terms are accounted for in the objective functional;
cf., e.g., \cite{Kunisch2015,Kunisch2015a,Nicaise2014}.
The fact that the problem is posed on a variable time-horizon introduces a
nonlinear dependency on the additional control variable \(T\). This significantly
complicates the analysis and numerical realization of~\eqref{P} compared to linear-quadratic 
problems with a fixed \(T\); see, e.g, \cite{Meidner2008a,Meidner2008,Meidner2011a}.
The goal of this article is to describe an appropriate fully space-time discrete
formulation, which is based on a transformation to a reference interval, and to prove 
optimal order \emph{a~priori} discretization error estimates.

Although time-optimal control is considered to be a classical subject in
control theory, to the best of our knowledge there are only a few publications concerning
the numerical analysis of such problems in the context of parabolic equations. The existing
contributions have in common that the terminal set is given by an $L^2$-ball around a
desired state (often assumed to be zero), the objective functional is $j(T,q) = T$, and
the state is discretized only in space by means of continuous linear finite elements.
In \cite{Schittkowski1979} convergence of optimal times for a one dimensional heat equation
is proved based on a bang-bang principle. Thereafter, more
general spatial domains have been considered. Purely time-dependent controls
acting on the boundary have been considered in \cite{Knowles1982}. There,
an error estimate for the optimal times of order
$\mathcal{O}(h^{3/2-\varepsilon})$ is proved for all $\varepsilon > 0$, assuming that $u_0 \in H^{3/2}(\Omega)$.
Furthermore, convergence of optimal times and controls for $u_d \neq 0$ with $u_0, u_d \in
H^{1/2-\varepsilon}(\Omega)$ is shown in \cite{Lasiecka1984} for a setting with
boundary control.
More recently, for distributed control and $u_0 \in H^1_0(\Omega)$ the error
estimate $\mathcal{O}(h)$ has been proved in \cite{Wang2012} for the linear heat
equation and for a semilinear heat equation in \cite{Zheng2014}. Both articles use
a cellwise linear discretization for the control and the set of admissible controls
is defined by
$\Qad \ldef \set{q \in L^\infty((0,\infty); L^2(\omega))
  \constraintSet  \norm{q(t)}_{L^2} \leq 1\; \mbox{a.e.}\; t}$.
Employing a variational control discretization, the error estimates $\mathcal{O}(h)$ for $T$ and
$\mathcal{O}(h^{1-\varepsilon})$ for the control and the state have been shown in \cite{Gong2016}.
Convergence of optimal times and controls for a class of abstract evolution equations has
been recently shown in \cite{Tucsnak2016}. We point out that the authors impose less
regularity on the initial value as in the references before, which in our setting would
correspond to the assumption $u_0 \in L^2(\Omega)$.
        
To the best of our knowledge this paper provides the first systematic numerical analysis for 
the full discretization of a time-optimal control problem.
This is one of the main novelties compared to the contributions mentioned above,
where only semidiscretizations (in space) have been considered.
Our approach is based upon a transformation to a reference interval.
The state equation is discretized by means of the
discontinuous Galerkin scheme in time (corresponding to a version of the implicit Euler
method) and linear finite elements in space. We
prove optimal convergence rates for the control variable for different control
discretization strategies. For example, in case of the variational control discretization
we obtain the convergence rate $\mathcal{O}(k + h^2)$ in all variables up to a logarithmic term. Here,
$k$ and $h$ denote the temporal and spatial mesh size, respectively.
We note that the presence of the cost term with $\alpha > 0$ is
crucial for our analysis and changes the character of the optimal solutions compared to the
classical case with \(\alpha = 0\). While in the latter case
we expect bang-bang controls containing jump-discontinuities, in the former case the
optimal controls are more regular. Nevertheless, this case is also interesting,
since it arises in the presence of control costs or if bang-bang controls are not
desirable. Moreover, it can be interpreted as a regularization strategy for the purely
time-optimal problem; cf.\ also \cite{Ito2010,Kunisch2013}. In this context, the
behavior of the discrete solutions as $\alpha\to 0$ has to be investigated. However, this
is beyond the scope of this article.

The convergence result is proved in two steps. First, we obtain a suboptimal convergence
rate for the control variable, where we rely on a quadratic growth condition that follows
from a second order sufficient optimality condition (SSC). Conceptually the
discretization error is related to differences of the objective functional for the
continuous and the discrete solutions, where we have to take square roots in the end; see
\cref{prop:errorEstimatesControlLocal}.
In the context of pointwise state constraints this is often acceptable, as low regularity of
the problem prevents better order convergence; cf., e.g., \cite{Neitzel2015}. However,
the solutions of~\eqref{P} exhibit better regularity, so we can expect an improved rate of
convergence. For the proof we adapt ideas from \cite{Casas2012d} for unconstrained
problems to the constrained case. This second estimate uses the SSC directly and relates
the discretization error to differences of derivatives of the
Lagrange function, which avoids taking square roots at the end; see \cref{lemma:ImprovedErrorRate}.
This directly results in the previously described $\mathcal{O}(\abs{\log k}(k + h^2))$ convergence result
using the variational discretization concept.
Note that this immediately implies the same result for the practically relevant
case of control by a finite number of time-dependent parameters; see
\cref{cor:optimalconvergence_parameter}.
For distributed controls we also consider an additionally discretization in space by
element-wise constant discontinuous or linear continuous finite elements. Here,
we obtain optimal rates of convergence that are additionally restricted by the
discrete space and the limited smoothness of the control variables, which may have discontinuous
derivatives.

As evident from the discussion above, the convergence result relies on the
SSC. In general, it is difficult to verify that for a given problem a SSC is
satisfied. In this regard, we note that SSCs have been used in
related contexts by many authors; see, e.g., \cite{Casas2012d,Neitzel2015}.
For the problem under consideration in this paper, we show that the SSC is equivalent to a scalar
condition that can be evaluated for a given optimal solution by solving an additional
linear quadratic optimization problem; see \cref{sec:scalar_ssc}.
Moreover, this scalar condition can be related to the curvature of a value
function, which arises from~\eqref{P} by resolving the corresponding linear quadratic
optimization problem for each fixed \(T\); cf.~\cite{Kunisch2013}.
This connection highlights the intrinsic importance of the SSC for the class
of optimization problems under consideration. Additionally, a similar
computation on the discrete level allows to compute this curvature constant for the
discrete problems with a small numerical effort. We consider this
to be an indicator for the SSC on the continuous level. In the numerical examples 
we observe that the curvature constant is bounded from below for
different cost parameters $\alpha$ and for sequences of refined discretizations uniformly
with respect to the mesh parameters \(k\) and \(h\).

This paper is organized as follows. In \cref{sec:notation_assumptions} we introduce the
notation and state the main assumption. First order and second order optimality conditions
for~\eqref{P}, which form the basis of the error analysis, are discussed in \cref{sec:optimalcontrolproblem}.
\Cref{sec:discretization} introduces the space-time discretization of~\eqref{P} and the
main convergence results are derived. Last, in \cref{sec:numerical_examples}, numerical
examples are given, which illustrate the convergence rates in the context of concrete
examples.

\section{Notation and main assumptions}\label{sec:notation_assumptions}  
For $\Omega \subset \R^d$ a Lipschitz domain, $H^1_0(\Omega)$ is the usual Sobolev space
with zero trace and the corresponding dual space is denoted by $H^{-1}(\Omega)$. The
duality pairing between $H^1_0(\Omega)$ and $H^{-1}(\Omega)$ is denoted
$\pair{\cdot,\cdot}$. If ambiguity is not to be expected, we drop the spatial domain
$\Omega$ from the notation of the spaces.
For $Z$ a Hilbert space, $\inner{\cdot,\cdot}_Z$ stands for its inner product. If $A$ is a linear operator on a Banach space $X$, we use $\dom{X}{A}$ to denote the domain of $A$ on $X$ equipped with the graph norm as usual.
Last, $c$ is a generic constant that may have different values at different appearances.

Throughout this paper we impose the following assumptions.
\begin{assumption}\label{assumption:Omega}
  Let $\Omega \subset \R^d$, $d \in \{\,2,3\,\}$, be a polygonal/polyhedral and
  convex domain, and $\alpha > 0$ a given cost
  parameter. 
  The initial value satisfies $u_0 \in H^1_0(\Omega)$.
\end{assumption}
Concerning the control operator $\ControlOp$ we consider one of the following
situations:
\begin{enumerate}[(i)]
\item Distributed control: Let $\omega \subseteq \Omega$ be the control domain that is polygonal or polyhedral as well. The control operator $\ControlOp \colon L^2(\omega) \to L^2(\Omega)$ is the extension by zero operator. Clearly, its adjoint $\ControlOp^* \colon L^2(\Omega) \to L^2(\omega)$ is the restriction to $\omega$ operator. 
\item Purely time-dependent control: For $N_c \in \N$, let $\omega = \set{1,2,\ldots,N_c}$
  equipped with the counting measure. The control operator is defined by $Bq =
  \sum_{n=1}^{N_c}q_n e_n$, where $e_n \in L^2(\Omega)$ are given form functions. Then
  we have $L^2(\omega) \cong \R^{N_c}$ and $B^* \colon L^2(\Omega) \to \R^{N_c}$ with $(B^*\varphi)_n =
  \inner{e_n,\varphi}_{L^2(\Omega)}$ for $n = 1,2\ldots,N_c$.
\end{enumerate}
The space of admissible controls is defined as
\begin{equation*}
  \Qad \ldef \left\{q \in \LcontrolSpatial \constraintSet  q_a \leq q \leq q_b \;\mbox{~a.e.\ in~}\; \omega\right\} \subset \LcontrolSpatial[\infty]
\end{equation*}
for $q_a, q_b \in \R$ with $q_a < q_b$.
In addition, for $T > 0$ set $\Q(0,T) \ldef L^2((0,T)\times\omega)$ and
\begin{equation*}
\Qad(0,T) \ldef \left\{q \in \Q(0,T) \constraintSet  q(t) \in \Qad\mbox{~a.e.~} t\in (0,T)\right\} \subset \Lcontrol[(0,T)]{\infty}.
\end{equation*}	
Moreover, we use $W(0,T)$ to abbreviate
$\MPRHilbert{(0,T)}{H^{-1}}{H^1_0}$, endowed with the canonical norm and inner product. The
symbol $i_T \colon W(0,T) \rightarrow H$ denotes the continuous trace mapping $i_T u = u(T)$.
We also define the canonical extension of the control operator $\ControlOp \colon
\Q(0,T) \to L^2((0,T)\times\Omega)$ by setting
$(\ControlOp q)(t) = \ControlOp q(t)$ for any $q \in \Q(0,T)$.

\begin{assumption}\label{assumption:terminal_constraint}
	The terminal constraint $G$ is defined by~\eqref{eq:terminal_constraint}
	for a fixed desired state $u_d \in H_0^1(\Omega)$ and $\delta_0> 0$.
\end{assumption}
\begin{remark}\label{remark:regularity_G}
	\begin{enumerate}[(i)]
		\item The error analysis remains valid for more general terminal constraints.
		Precisely, we require that $G$ is two times continuously
                Fr{\'e}chet-differentiable, the mapping $\eta \mapsto G''(u)[\eta]^2$ is
                weakly lower semicontinuous, $G''$ is bounded on bounded sets in
                $L^2(\Omega)$, and $G'(u)^* \in H^1_0$ for any $u\in H^1_0$.
		We restrict attention to~\eqref{eq:terminal_constraint} in order to make
                the main ideas more transparent to the reader.
		Another terminal constraint that would fit into this more general setting can be found in \cite[Section~5.4]{Bonifacius2017}.
		\item The regularity assumption $u_d \in H_0^1(\Omega)$ is required for optimal order of convergence.
		Since $G'(u)^* = u-u_d$ defines the terminal value of the adjoint equation, this leads to improved regularity of the adjoint equation, which in turn allows to prove full order of convergence.
		\item In addition, we would like to justify the regularity assumption $u_d \in
		H_0^1(\Omega)$ from a different perspective,
                namely that of weak invariance. The target set $U = \set{u\in L^2(\Omega) \constraintSet  G(u) \leq 0}$ is called
		\emph{weakly invariant} under the state equation if, for any $u_0$ satisfying $G(u_0) \leq 0$, there
		is a admissible control \(q(t)\in\Qad\) such that the corresponding trajectory with initial value
		$u_0$ satisfies $G(u(t)) \leq 0$ for all times; cf.\
		\cite[Section~4]{Bonifacius2017} and the references therein.
		Since the formulation of~\eqref{P} only requires the
		state to be inside the target set at the final time \(T\) (but not at
                later times), it seems to be
                desirable to require the target set to be weakly invariant, since this
                guarantees that \(G(u(t)) \leq 0\) can be maintained for \(t > T\).
                However, this requirement already implies that the metric projection
                \(P_U\) to \(U\) in \(L^2(\Omega)\) is stable in \(H^1_0(\Omega)\);
                see~\cite[Lemma~3.5]{Bonifacius2017}. This further leads to
                the requirement $G'(P_U(u))^* = P_U(u)-u_d \in
                H_0^1(\Omega)$ for all $u \in H_0^1(\Omega)$, which implies the 
                assumption on \(u_d\).
	\end{enumerate}	
\end{remark}

In order to ensure existence of feasible points, we require the following:
\begin{assumption}\label{assumption:existence_feasible_control}
	There exist a finite time $T > 0$ and a feasible control $q \in \Qad(0,T)$ such that the solution to the state equation of \eqref{P} satisfies $G(u(T)) \leq 0$. To exclude the trivial case, we additionally assume $G(u_0) > 0$.
\end{assumption}

\begin{remark}
  We exemplary state situations where \cref{assumption:existence_feasible_control} holds.
  \begin{enumerate}[(i)]
  \item
    In case of distributed control on an open subset $\omega \subset \Omega$, the
    state equation is known to be approximately controllable (see, e.g.,
    \cite{Rosier2007,Zuazua2007}).
    This guarantees existence of feasible controls for sufficiently large
    control constraints relative to \(u_d\); cf.\ also~\cite{Fernandez-Cara2000} for
    estimates on the necessary size of the control bounds.
  \item
    For $u_d = 0$ and if $0 \in \Qad(0,1)$, then for any
    $\delta_0 > 0$ the control \(q\equiv 0\) is feasible for $T > 0$ sufficiently large,
    since the semigroup generated by $\upDelta$ is exponentially stable in
    $L^2(\Omega)$.
  \item
    To generalize the previous statement, assume there
    is a control $\breve{q} \in \Qad$ with
    \begin{equation}\label{eq:qualified_opt_sufficient_Laplace}
      \norm{B\breve{q}+\Lap u_d}_{H^{-1}} < \frac{c_P^2}{1+c_P^2}\,\delta_0,
    \end{equation}
    where $c_P$ denotes the Poincar{\'e} constant. Then, \(q \equiv \breve{q}\) is a feasible control
    for large enough \(T\); see~\cite[Lemma~3.9, Proposition~5.3]{Bonifacius2017}.
  \end{enumerate}
\end{remark}

\section{Optimal control problem}\label{sec:optimalcontrolproblem}
Since the problem~\eqref{P} is posed on a variable time-domain, we first introduce a
transformation to the unit time interval, which is the basis for the subsequent analysis
and the numerical methods.

\subsection{Change of variables}
For $\nu \in \Rplus$ we perform a change of variable \(t \mapsto \nu t\)
and obtain the transformed state equation
\begin{equation*}
\partial_t u(t) - \nu\Lap u(t) = \nu\ControlOp q(t)\mbox{,}\quad
t \in (0,1),\quad
u(0) = u_0\mbox{.}
\end{equation*}
For the transformed state equation on the unit time interval $I = (0,1)$, the parameter \(\nu\)
replaces the free end time \(T\).
Standard results for parabolic equations (see, e.g.,~\cite[Theorem~2, Chapter XVIII, \S3]{Dautray1992}) imply that for each pair $(\nu,q) \in \Rplus\times\Q(0,1)$ there exists a unique solution to the transformed state equation.
Let $S \colon \Rplus\times\Q(0,1) \rightarrow W(0,1)$, $(\nu,q)\mapsto u$ denote the corresponding control-to-state mapping. 
We endow the product space $\R\times\Lcontrol{2}$ with the canonical inner product 
and abbreviate its norm as
\begin{equation*}
\prodnorm{\delta\nu, \delta q} = \left(\abs{\delta\nu}^2 + \norm{\delta q}^2_{\Lcontrol{2}}\right)^{1/2}.
\end{equation*}
For convenience of notation, we sometimes abbreviate \(\chi = (\nu,q)\).
Moreover, we introduce the reduced objective and constraint functionals as
\begin{equation*}
g(\nu,q) \ldef G(i_1S(\nu,q))\mbox{,}\quad
j(\nu,q) \ldef \int_0^1 \nu \left(1+\frac{\alpha}{2}\norm{q(t)}_{\LcontrolSpatial}^2\right) \D{t}\mbox{.}
\end{equation*}	
The transformed optimal control problem is then given by
\begin{align}\label{Pt}\tag{\mbox{$\hat{P}$}}
\inf_{\substack{\nu\in\Rplus\\ q \in \Qad(0,1)}}&j(\nu, q) \quad\mbox{subject to}\quad g(\nu,q) \leq 0\mbox{.}
\end{align}
The definition of the set of admissible controls $\Qad$ transfers to the transformed problem, because of time
independence of the control constraints. In fact, both
problems~\eqref{Pt} and~\eqref{P} are equivalent; cf., e.g., \cite[Proposition~4.6]{Bonifacius2017}.
Because no ambiguity arises between~\eqref{Pt} and~\eqref{P}, we do not rename variables.

Since there exists at least one feasible control due to
\cref{assumption:existence_feasible_control}, well-posedness of~\eqref{Pt} is obtained by standard
arguments; cf., e.g., \cite[Proposition~4.1]{Bonifacius2017}. Note, that
\(\nu = 0\) is not admissible due to the assumption \(G(u_0) > 0\), and that the optimal
solution must fulfill the terminal constraint with equality (otherwise, a control with a
shorter time is still admissible, while having a smaller objective value).
\begin{proposition}
  Problem~\eqref{Pt} admits a solution $(\bar{\nu}, \bar{q}) \in \Rplus\times\Qad(0,1)$ with
  associated state \(\bar{u} = S(\bar{\nu},\bar{q})\).
  Moreover, it holds \(g(\bar{\nu},\bar{q}) = G(\bar{u}(1)) = 0\).
\end{proposition}

We now give several simple auxiliary results, which will be needed throughout the paper. 
First, by standard arguments, we obtain the following differentiability result.
\begin{lemma}\label{lemma:control_to_state_differentiable}
	Let $\nu \in \Rplus$ and $q \in \Q(0,1)$.
	The control-to-state mapping $S$ is twice continuously Fr\'echet-differentiable. Moreover, $\delta u = S'(\nu,q)(\delta \nu, \delta q) \in W(0,1)$ is the unique solution to
	\begin{equation*}
	\partial_t \delta u - \nu\Lap\delta u = \delta\nu (\ControlOp q+\Lap u)+ \nu \ControlOp\delta q\mbox{,}\quad
	\delta u(0) = 0\mbox{,}
	\end{equation*}
	for $(\delta \nu, \delta q) \in \R\times\Lcontrol{2}$ and $\delta\tilde{u} = S''(\nu,q)(\delta\nu_1, \delta q_1;\delta \nu_2, \delta q_2) \in W(0,1)$ is the unique solution to
	\begin{equation*}
	\partial_t \delta\tilde{u} - \nu\Lap\delta\tilde{u} = \delta\nu_1\left(\ControlOp\delta q_2 + \Lap\delta u_2\right)+ \delta\nu_2 \left(\ControlOp\delta q_1 +\Lap\delta u_1\right)\mbox{,}\quad
	\delta\tilde{u}(0) = 0\mbox{,}
	\end{equation*}
	for $(\delta \nu_i, \delta q_i) \in \R\times\Lcontrol{2}$ and $\delta u_i = S'(\nu,q)(\delta \nu_i, \delta q_i)$, $i = 1,2$.
\end{lemma}

By means of \cref{lemma:control_to_state_differentiable}, the reduced constraint mapping
$g\colon \Rplus\times\Q(0,1) \rightarrow \R$ is twice continuously
Fr{\'e}chet-differentiable. Moreover, the derivatives can be computed by the chain-rule.
\begin{proposition}\label{proposition:derivatives_g}
For any \((\nu,q) \in \Rplus\times \Q(0,1)\)
and $u = S(\nu,q)$, we have
\begin{align*}
g'(\nu,q)(\delta{\nu}, \delta{q})
 &= \left(u(1) - u_d,\delta{u}(1)\right), \\
g''(\nu,q)(\delta{\nu}_1, \delta{q}_1; \delta{\nu}_2, \delta{q}_2)
 &= \left(\delta{u}_1(1),\delta{u}_2(1)\right) + \left(u(1) - u_d,\delta{\tilde{u}}(1)\right),
\end{align*}
where \(\delta{u}_1\), \(\delta{u}_2\), and \(\delta{\tilde{u}}\) are defined as in
\cref{lemma:control_to_state_differentiable}.
\end{proposition}
Based on this, we can derive a continuity result on the second derivative.
\begin{corollary}\label{corollary:weak_lower_semicontinuity_g}
	Let $(\nu,q) \in \Rplus\times\Q(0,1)$.
	If $\delta\nu_n \rightarrow \delta\nu$ in $\R$ and $\delta q_n \rightharpoonup \delta q$ weakly in $\Lcontrol{2}$, then
	\begin{equation*}
	g''(\nu, q)[\delta\nu, \delta q]^2 \leq \liminf_{n\rightarrow \infty} g''(\nu, q)[\delta\nu_n, \delta q_n]^2\mbox{.}
	\end{equation*}
\end{corollary}
\begin{proof}
We use the structure of the derivative, see \cref{proposition:derivatives_g}, and
verify that
\begin{align*}
S'(\nu,q)(\delta\nu_n, \delta q_n) &\rightharpoonup S'(\nu,q)(\delta\nu, \delta q)\quad \mbox{in~}W(0,1),\\
S''(\nu,q)[\delta\nu_n, \delta q_n]^2 &\rightharpoonup S''(\nu,q)[\delta\nu, \delta q]^2\quad \mbox{in~}W(0,1),
\end{align*}
due to the bilinear structure. Using the fact that the trace mapping \(i_1\) is
continuous on \(W(0,1)\), we infer the result.
\end{proof}

Moreover, a formula for the gradient of the constraint functional
can be derived based on the adjoint approach.
To avoid confusion with the spatial gradient \(\nabla\), we denote the
gradient of \(g\) by \(g'(\cdot)^*\) in the following.
\begin{proposition}
For $\nu \in \Rplus$, $q \in \Q(0,1)$, \(u = S(\nu,q)\), and $\mu \in \R$ we have the representation
\begin{equation}\label{eq:terminal_constraint_adjoint}
\mu \, g'(\nu,q)^* = \left(
\begin{array}{l}
  \int_{0}^{1}\pair{Bq + \Lap u, z}\\
  \nu B^*z
\end{array} \right)\mbox{,}
\end{equation}
where $z \in W(0,1)$ is the unique solution to the dual equation
\begin{equation*}
-\partial_t z - \nu\Lap z = 0\mbox{,}\quad 
z(1) = \mu(u(1) - u_d)\mbox{;}
\end{equation*}
\end{proposition}
\begin{proof}
The result can be derived as in, e.g., \cite[Proposition~4.8]{Bonifacius2017}.
\end{proof}
Finally, for \(\nu\) bounded uniformly from below and above, the derivatives of \(g\) can
be estimated by uniform constants, which will be important in the following.
\begin{proposition}\label{prop:stability_g} 
	Let $0 < \nu_{\min} < \nu_{\max}$ be given. Then there exists $c > 0$ such that for all $\delta\nu \in \R$ and $\delta q \in \Lcontrol{2}$ it holds
	\begin{align*}
	\abs{g'(\nu, q)(\delta\nu, \delta q)} & \leq c\prodnorm{\delta\nu, \delta q}\mbox{,}\\
	\abs{g''(\nu, q)[\delta\nu, \delta q]^2}  &\leq c \prodnorm{\delta\nu, \delta q}^2\mbox{,}		
	\end{align*}
	for all $\nu_{\min} \leq \nu \leq \nu_{\max}$ and $q \in \Qad(0,1)$.
	Moreover,
	\begin{equation*}
	\abs{\left(g'(\nu_1,q_1) - g'(\nu_2,q_2)\right)(\delta\nu, \delta q)} \leq c\prodnorm{\nu_1-\nu_2,q_1-q_2}\prodnorm{\delta\nu, \delta q}\mbox{,}
	\end{equation*}
	for all $\nu_{\min} \leq \nu_1, \nu_2 \leq \nu_{\max}$ and $q_1, q_2 \in \Qad(0,1)$.
\end{proposition}
\begin{proof}
	Since $g(\nu,q) = G(i_1S(\nu,q))$
	the result is a consequence of the stability properties of \(S\) (see~\cref{prop:stabilityS}, where also the precise
        dependency of the constants on \(\nu_{\min}\) and \(\nu_{\max}\) is given) and the structure of $G$.
\end{proof}

\subsection{First order optimality conditions}
The numerical analysis essentially relies on first and second order optimality conditions.
We start by discussing first order necessary conditions; see also \cite{Raymond1999}.
To this end, let $(\bar{\nu},\bar{q})$ be a locally optimal control for~\eqref{P}.
We require the following \emph{linearized Slater} condition.
\begin{assumption}
\label{assumption:linearized_slater}
We assume that
\begin{equation}\label{eq:definition_linearized_slater_condition}
  \bar{\eta} \ldef - \partial_\nu g(\bar{\nu}, \bar{q}) > 0\mbox{.}
\end{equation}
\end{assumption}
Note that by \cref{assumption:linearized_slater} and \(g(\bar{\nu}, \bar{q}) = 0\), the
point \(\breve{\chi}^\gamma = (\bar{\nu} + \gamma, \bar{q}) \in \Rplus\times\Qad(0,1)\)
defined for \(\gamma > 0\) fulfills
\begin{equation}\label{eq:definition_linearized_slater_condition_general}
  g(\bar{\chi}) + g'(\bar{\chi})(\breve{\chi}^\gamma - \bar{\chi}) = -\bar{\eta}\,\gamma < 0\mbox{,}
\end{equation}
which corresponds to a more familiar presentation of the linearized Slater condition. Thus, we essentially
assume this condition to hold in a special form. Roughly speaking,
we require the terminal constraint to decrease sufficiently when the time
horizon is enlarged over the optimal time.
However, we will see that, for the particular problem at hand,
\cref{assumption:linearized_slater} is already equivalent to qualified first order
conditions, and thus essentially equivalent to any other constraint qualification.

In order to state optimality conditions, we introduce the Lagrange function as
\begin{equation*}
\mathcal{L} \colon \Rplus\times\Q(0,1)\times\R \rightarrow \R\mbox{,}\quad
\mathcal{L}(\nu, q, \mu) \ldef j(\nu, q) + \mu\, g(\nu, q)\mbox{.}
\end{equation*}
Now, optimality conditions for~\eqref{Pt} in qualified form can be stated as follows: 
for given \(\bar{\nu} > 0\) and \(\bar{q} \in \Qad(0,1)\) with \(g(\bar{\nu},\bar{q})=0\)
there exists a \(\bar{\mu}\geq 0\), such that
\begin{equation}\label{eq:optimalityCondLagrange}
\partial_{(\nu,q)}\mathcal{L}(\bar{\nu},\bar{q},\bar{\mu})(\delta{\nu},q-\bar{q}) \geq 0
\quad \text{for all }(\delta{\nu},q) \in \R\times\Qad(0,1)\mbox{.}
\end{equation}
With \cref{assumption:linearized_slater}, a multiplier always exists and, due
to the special structure, it is always positive. We summarize this in the next result.
\begin{lemma}\label{lemma:first_order_optcond}
	Let $(\bar{\nu},\bar{q}) \in \Rplus\times\Qad(0,1)$ be a solution of~\eqref{P}
        with associated state $\bar{u} = S(\bar{\nu},\bar{q})$ and the linearized Slater
        condition~\eqref{eq:definition_linearized_slater_condition} hold. Then there exists
        a multiplier $\bar{\mu} \in (\,0,\, c/\bar{\eta}\,] \subset \Rplus$ such that
	\begin{align}
	\int_{0}^{1} 1 + \frac{\alpha}{2}\norm{\bar{q}(t)}_{\LcontrolSpatial}^2 + \pair{\ControlOp\bar{q}(t) + \Lap\bar{u}(t), \bar{z}(t)}\D{t} &= 0\mbox{,}\label{eq:opt_cond_hamiltonianConstant}\\
	\int_0^{1} \bar{\nu}\pair{\alpha \bar{q}(t)+\ControlOp^*\bar{z}(t), q(t) - \bar{q}(t)}\D{t} &\geq 0\mbox{,} & q &\in\Qad(0,1)\mbox{,}\label{eq:opt_cond_variationalInequality}\\
	G(\bar{u}(1)) &= 0\mbox{,}\label{eq:opt_cond_feasiblity}
	\end{align}
	where the \emph{adjoint state} $\bar{z} \in W(0,1)$ is determined by
	\begin{equation}\label{eq:adjoint_state_equation}
	-\partial_t \bar{z}(t) - \bar{\nu}\Lap\bar{z}(t) = 0\mbox{,}
        \quad t \in (0,1) \quad 
	\bar{z}(1) = \bar{\mu}(\bar{u}(1) - u_d)\mbox{.}
	\end{equation}
\end{lemma}
\begin{proof}
	We first note that the linearized Slater condition allows for exact penalization of \eqref{Pt}; see 
	\cite[Theorem~2.87, Proposition~3.111]{Bonnans2000}.
	The optimality conditions now follow as in the proof of \cite[Theorem~4.12]{Bonifacius2017}.
        The condition~\eqref{eq:opt_cond_hamiltonianConstant} is equivalent to
        \(\partial_\nu \mathcal{L}(\bar{\nu},\bar{q},\bar{\mu}) = 0\)
        and~\eqref{eq:opt_cond_variationalInequality} arises
        from~\eqref{eq:optimalityCondLagrange} for \(\delta{\nu} = 0\).
	Note that \(\bar{\mu} = 0\) implies \(\bar{z} = 0\), which
        contradicts~\eqref{eq:opt_cond_hamiltonianConstant}.
        Thus \(\bar{\mu} > 0\) must hold.
\end{proof}

The optimality condition for the free end time~\eqref{eq:opt_cond_hamiltonianConstant}
allows to prove equivalence of qualified optimality conditions and
condition~\eqref{eq:definition_linearized_slater_condition}.
\begin{proposition}\label{prop:qualifedOptCndsIffCQ}
	The qualified first order optimality conditions of \cref{lemma:first_order_optcond} hold
	if and only if~\eqref{eq:definition_linearized_slater_condition} is valid.
\end{proposition}
\begin{proof}
  Assume the first order conditions to hold.
  According to~\eqref{eq:optimalityCondLagrange} we have
  \begin{align*}
    \bar{\mu}\,\partial_\nu g(\bar{\nu},\bar{q})
    & = - \partial_\nu j(\bar{\nu},\bar{q})
    = - \int_0^1\left(1+\frac{\alpha}{2}\norm{\bar{q}(t)}^2\right)\D{t}
      \leq - 1\mbox{.}
  \end{align*}
  Hence, condition~\eqref{eq:definition_linearized_slater_condition} holds with
  \(\bar{\eta} \geq 1/\bar{\mu} > 0\). The remaining
  implication is the assertion of \cref{lemma:first_order_optcond}.
\end{proof}

Using $\bar{\nu} > 0$, we derive from~\eqref{eq:opt_cond_variationalInequality} the usual projection formula:
\begin{equation}\label{eq:projectionFormulaControl}
\bar{q} = \ProjQad{-\frac{1}{\alpha}\ControlOp^*\bar{z}}\mbox{,}
\end{equation}
where $\ProjQad{\cdot}$ denotes the pointwise projection onto the set $\Qad$, defined by
\begin{equation*}
\ProjectionQad \colon \Q(0,1) \rightarrow \Qad(0,1)\mbox{,}\quad
\ProjectionQad(r)(t,x) = \max \left\{q_a, \min\left\{q_b, r(t,x)\right\}\right\}\mbox{.}
\end{equation*}
In particular, it holds (almost everywhere) in $I\times\omega$ that:
\begin{equation}\label{eq:optimalitySignConditionControlPlusAdjoint}
\begin{cases}
\bar{q}(t, x) = q_a &\mbox{if~} \alpha \bar{q}(t,x)+\ControlOp^*\bar{z}(t,x) > 0\mbox{,}\\
\bar{q}(t, x) = q_b &\mbox{if~} \alpha \bar{q}(t,x)+\ControlOp^*\bar{z}(t,x) < 0\mbox{.}\\
\end{cases}
\end{equation}
From this, we obtain additional regularity, which will be used for the error estimates.
\begin{proposition}\label{prop:regularityOptimalSolution}
	The optimal state $\bar{u}$ and the adjoint state
	$\bar{z}$ to~\eqref{Pt} exhibit the improved regularity               
	\begin{equation*}
	\bar{u}, \bar{z} \in \MPRHilbert{I}{L^2}{H^2\cap H_0^1} \embedding C([0,1]; H_0^1)\mbox{.}
	\end{equation*}
	Additionally, in case of distributed control we have
	\[
	\bar{q} \in \MPRHilbert{I}{L^2(\omega)}{H^1(\omega)}\mbox{.}
	\]
\end{proposition}
\begin{proof}
	Since $\Omega$ is convex, elliptic regularity yields $\dom{L^2}{-\Lap} = H^2\cap H^{1}_0$;
	see, e.g., \cite[Theorem~3.2.1.2]{Grisvard1985}.
	Hence, the assertion follows from standard regularity theory for 
	the heat equation (see, e.g., \cite[Theorem~7.1.5]{Evans2010})
	and the projection formula~\eqref{eq:projectionFormulaControl}.	
\end{proof}

Finally, we would like to highlight a concrete situation where the optimality conditions
\eqref{eq:opt_cond_hamiltonianConstant}--\eqref{eq:opt_cond_feasiblity} (equivalently
\cref{assumption:linearized_slater}) are guaranteed to hold.
\begin{theorem}
Assume that for the given \(u_d\) and \(\delta_0\) there is a control \(\breve{q} \in \Qad\),
such that \eqref{eq:qualified_opt_sufficient_Laplace} holds. Then,
\cref{assumption:linearized_slater} holds with \(\bar{\eta} \geq
\eta_{\min}(\delta_0,u_d,\Qad)\) for any optimal solution \((\bar{\nu},\bar{q})\).
\end{theorem}
\begin{proof}
Condition~\eqref{eq:qualified_opt_sufficient_Laplace} is sufficient for qualified
optimality conditions, with multiplier \(\bar{\mu}\) bounded uniformly
only in terms of \((\delta_0,u_d,\Qad)\); see~\cite[Theorem~4.12, Proposition~5.3]{Bonifacius2017}.
Moreover, as in the proof of~\cref{prop:qualifedOptCndsIffCQ} it can be verified that for
any optimal solution it holds \(\bar{\eta} \geq 1/\bar{\mu}\), which is uniformly bounded
from below.
\end{proof}

\subsection{Second order optimality conditions}
Since~\eqref{Pt} is a nonconvex optimization problem, first order optimality conditions are not sufficient for optimality. We therefore discuss second order optimality conditions employing a cone of critical directions, introduced as
\begin{equation*}
\CriticalCone = \left\{(\delta\nu, \delta q) \in \R\times \Lcontrol{2} \,\left| 
\begin{aligned}
\delta q \mbox{~satisfies the sign condition~\eqref{eq:secondOrderSignCondition}, and}\\
g'(\bar{\nu},\bar{q})(\delta\nu, \delta q) = 0
\end{aligned}\right.
\right\}\mbox{,}
\end{equation*}
where the sign condition is given by
\begin{equation}
\delta q(t,x) \left\{
\begin{aligned}
\leq 0 &\mbox{~if~} \bar{q}(t,x) = q_b \\
\geq 0 &\mbox{~if~} \bar{q}(t,x) = q_a\\
= 0 &\mbox{~if~} \alpha\bar{q}(t,x)+\ControlOp^*\bar{z}(t,x) \neq 0
\end{aligned}\right\}\quad \mbox{a.e. in~}I\times\omega\mbox{.}\label{eq:secondOrderSignCondition}
\end{equation}
With this definition, we can formulate second order necessary conditions, which
hold in any locally optimal stationary point.
\begin{theorem}\label{thm:secondOrderNecessary}
	Let $(\bar{\nu},\bar{q}) \in \Rplus\times\Qad(0,1)$ be a local minimum of~\eqref{Pt} and $\bar{\mu} > 0$ satisfying first order optimality conditions of \cref{lemma:first_order_optcond}. Then
	\begin{equation*}
	\partial_{(\nu,q)}^2\mathcal{L}(\bar{\nu},\bar{q},\bar{\mu})[\delta\nu, \delta q]^2\geq 0 \quad\text{for all }(\delta\nu, \delta q) \in \CriticalCone\mbox{.}
	\end{equation*}
\end{theorem}
\begin{proof}
	The assertion can be proved similarly as in \cite{Casas2002}.
	According to the linearized Slater condition~\eqref{eq:definition_linearized_slater_condition}, we have 
	$g'(\bar{\nu},\bar{q})(\delta\breve{\chi}) = 1$ for \(\delta\breve{\chi} =
        (- 1/\bar{\eta}, 0)\). Hence, the regularity assumption \cite[equation~(2.1)]{Casas2002}
	is automatically satisfied in our setting.
\end{proof}
It is well-known that the condition~\eqref{thm:secondOrderNecessary} does not
suffice to derive optimal error estimates. Next, we postulate ``minimal-gap'' second order
sufficient conditions, which result from replacing the inequality
in~\eqref{thm:secondOrderNecessary} by a strict inequality.
\begin{theorem}\label{thm:secondOrderSufficient}
	Suppose $(\bar{\nu},\bar{q}) \in \Rplus\times\Qad(0,1)$ and $\bar{\mu} > 0$ satisfy the first order necessary condition of \cref{lemma:first_order_optcond} as well as the second order sufficient condition
	\begin{equation}\label{eq:secondOrderSufficientCondition}
	\partial_{(\nu,q)}^2\mathcal{L}(\bar{\nu},\bar{q},\bar{\mu})[\delta\nu, \delta q]^2 > 0\quad\text{for all }(\delta\nu, \delta q) \in \CriticalCone\setminus\{(0,0)\}\mbox{.}
	\end{equation}
	Then there exist $\varepsilon > 0$ and $\kappa > 0$ such that for every admissible pair $(\nu, q) \in \Rplus\times\Qad(0,1)$ the quadratic growth condition
	\begin{equation}\label{eq:quadraticGrowthCondition}
	j(\bar{\nu},\bar{q}) + \frac{\kappa}{2}\abs{\nu-\bar{\nu}}^2 + \frac{\kappa}{2}\norm{q-\bar{q}}_{\Lcontrol{2}}^2 \leq j(\nu,q)\mbox{,}
	\end{equation}
	is satisfied if $\abs{\nu-\bar{\nu}} + \norm{q-\bar{q}}_{\Lcontrol{2}} \leq \varepsilon$.
\end{theorem}
\begin{proof}
	The assertion can be proved similarly as in \cite[Theorem~4.13]{Casas2015}.
\end{proof}
The second order sufficient condition~\eqref{eq:secondOrderSufficientCondition} and
the quadratic growth condition~\eqref{eq:quadraticGrowthCondition} will form the
basis of the following analysis. Last, we note that for the given objective functional,
coercivity of $\partial_{(\nu,q)}^2\mathcal{L}(\bar{\nu},\bar{q},\bar{\mu})$ is equivalent
to the seemingly weaker positivity condition~\eqref{eq:secondOrderSufficientCondition}, as
already observed for semilinear parabolic PDEs in~\cite{Casas2015}.

\begin{theorem}\label{thm:secondOrderSufficientEquivalence} Let $(\bar{\nu},\bar{q}) \in \Rplus\times\Qad$ and $\bar{\mu} > 0$. The \emph{positivity condition}~\eqref{eq:secondOrderSufficientCondition} is equivalent to
	the \emph{coercivity condition}: there exists a \(\bar{\kappa} > 0\) such that
	\begin{equation*}
	\partial_{(\nu,q)}^2\mathcal{L}(\bar{\nu},\bar{q},\bar{\mu})[\delta\nu, \delta q]^2
        \geq \bar{\kappa} \left(\abs{\delta\nu}^2 + \norm{\delta q}_{\Lcontrol{2}}^2\right)\quad\text{for all }(\delta\nu,\delta q) \in \CriticalCone\mbox{.}
	\end{equation*} 
\end{theorem}
\begin{proof}
	This result can be proved along the lines of the proof of \cite[Theorem~4.11]{Casas2015}, where we in particular use \cref{corollary:weak_lower_semicontinuity_g}.
\end{proof}

\subsection{Characterization of the SSC}
\label{sec:scalar_ssc}
In general it seems to be difficult to verify whether a second order sufficient
optimality condition is satisfied for a given a problem -- both theoretically and
numerically.
However, for the problem under consideration here, we will provide a
scalar condition that is equivalent to
the second order sufficient optimality condition of~\cref{thm:secondOrderSufficient};
cf.\ \cite{Ito2010}
for a similar approach for time-optimal control of ODEs.	
The idea leads to a
simple test of the SSC based on the solution of one linear-quadratic auxiliary problem.

In order to keep the presentation of this section simple, we impose additional assumptions, which will
be fulfilled in most situations.
First, if the critical cone is trivial, i.e.\ $\CriticalCone = \set{0}$, the
condition~\eqref{eq:secondOrderSufficientCondition} is vacuously true.
Note that this case corresponds to a bang-bang control,
which can occur only if the control assumes either only the lower or upper bound on each
connected component of \(\omega\) (taking into account the projection
formula~\eqref{eq:projectionFormulaControl}).
Similarly, to avoid other degenerate cases, we impose the additional assumption: 
\begin{assumption}\label{assumption:non_degenerate}
We assume that the critical cone $\CriticalCone$ is a linear space that
contains elements of the form \((\delta{\nu},\delta{q})\) with \(\delta{\nu}\neq 0\).
\end{assumption}
\begin{remark}\label{remark:assumption_non_degenerate}
  \cref{assumption:non_degenerate} is equivalent to a strict complementarity condition and
  a non-triviality condition, given concretely by
  \begin{align*}
    \abs{\set{(t,x) \in I\times\omega \constraintSet  \bar{q}(t,x) \in \set{q_a, q_b},\,
    \alpha\bar{q}(t,x) + \ControlOp^*\bar{z}(t,x) = 0}} &= 0,\\
    \abs{\set{(t,x) \in I\times\omega \constraintSet   q_a < \bar{q}(t,x) < q_b,\,
    \ControlOp^*\bar{z}(t,x) \neq 0}} &> 0,
  \end{align*}
  where $\abs{\cdot}$ denotes the product-measure associated with $I\times\omega$.

  We note that it is possible to show that these assumptions are already equivalent to
  $\CriticalCone \neq \set{0}$, either in the setting of a distributed control, or under an
  approximate controllability assumption on \((-\Lap,B)\).
\end{remark}
     
If \cref{assumption:non_degenerate} holds, the critical cone consists exactly of the
elements \((\delta{\nu}, \delta{q})\) with \(\delta{\nu} \in \R\), \(\delta{q} \in
C_{\bar{q}}\), and \(\partial_q g(\bar{\nu},\bar{q})\delta{q} +
\partial_{\nu}g(\bar{\nu},\bar{q})\delta{\nu} = 0\), where 
\[
C_{\bar{q}} \ldef \set{\delta q \in \Lcontrol{2} \constraintSet  \delta q(t,x) = 0 \text{ if } \alpha\bar{q}(t,x)+\ControlOp^*\bar{z}(t,x) \neq 0}.
\]
For ease of presentation, we sometimes abbreviate the arguments $(\bar{\nu},\bar{q})$
and simply write $\bar{\chi}$ in the following.

\begin{lemma}\label{lemma:ssc_reduces_to_scalar_condition}
  Let $(\bar{\nu},\bar{q}) \in \Rplus\times\Qad(0,1)$
  and assume that \cref{assumption:non_degenerate} holds.
  The second order sufficient optimality condition of
  \cref{thm:secondOrderSufficient} is equivalent to
  \begin{equation}\label{eq:ssc_schur_scalar_condition}
    \bar{\gamma} \ldef \partial^2_{(\nu,q)} \mathcal{L}(\bar{\nu},\bar{q},\bar{\mu})[1,\delta\bar{q}]^2 > 0,
  \end{equation}
  where $(\delta\bar{q},\delta\bar{\mu}) \in C_{\bar{q}}\times\R$ is the unique solution
  of the linear system
  \begin{equation}
    \label{eq:ssc_reduces_to_scalar_condition_system}
    \begin{aligned}
      \partial^2_{q}\mathcal{L}(\bar{\nu},\bar{q},\bar{\mu})[\delta\bar{q},\delta q]
      + \delta\bar{\mu} \, \partial_q g(\bar{\nu},\bar{q})\delta q &= 
      -\partial_{\nu}\partial_q \mathcal{L}(\bar{\nu},\bar{q},\bar{\mu})[1,\delta q],\quad \delta q \in C_{\bar{q}},\\
      \partial_q g(\bar{\nu},\bar{q})\delta\bar{q}
      &= - \partial_{\nu} g(\bar{\nu},\bar{q}).
    \end{aligned}
  \end{equation}
\end{lemma}
\begin{proof}
Clearly, we only have to prove that \eqref{eq:ssc_schur_scalar_condition} implies the
second order sufficient optimality condition, since the other implication is obvious.
Let $(\delta\nu,\delta q) \in \CriticalCone$.
We distinguish two cases for \(\delta{\nu}\). If \(\delta{\nu} = 0\), we use the fact that
the second derivative of \(g\) with respect to \(q\) has the form
\(\partial_{q}^2g(\bar{\chi})[\delta{q}]^2 = \norm{i_1 \partial_{q}S(\bar{\chi})\delta{q}}^2_{L^2}\) to obtain
\begin{equation*}
\partial^2_{q}\mathcal{L}(\bar{\chi},\bar{\mu})[\delta{q}]^2
 \geq \partial^2_{q}j(\bar{\chi})[\delta{q}]^2 = \alpha\bar{\nu}\norm{\delta{q}}^2_{L^2(I\times\omega)},
\end{equation*}
which immediately implies~\eqref{eq:secondOrderSufficientCondition}.
Now, consider the case \(\delta{\nu} \neq 0\). Since the expression on the left
in~\eqref{eq:secondOrderSufficientCondition} is bi-linear in \(\delta{\nu}\), and the
critical cone \(\CriticalCone\) is linear, it suffices
to consider the case \(\delta{\nu} = 1\). By minimizing the expression on the left for
admissible \(\delta{q}\) (such that \((1,\delta{q}) \in \CriticalCone\)), writing out the
second derivative in terms of the partial derivatives and dropping constant terms, we
arrive at the following linear-quadratic minimization problem:
\begin{equation}\label{eq:ssc_equiv_scalar_aux_minimization}
  \inf_{\delta q \in C_{\bar{q}}}
  \frac{1}{2}\partial^2_{q}\mathcal{L}(\bar{\chi},\bar{\mu})[\delta q]^2
  + \partial_{\nu}\partial_{q}\mathcal{L}(\bar{\chi},\bar{\mu})[1,\delta q]
  \quad\text{subject to}\quad
 \partial_{q}g(\bar{\chi})\delta q = -\partial_{\nu} g(\bar{\chi}).
\end{equation}
Since $(1,\delta q) \in \CriticalCone$, we have $\partial_q g(\bar{\chi})\delta{q} = -\partial_\nu g(\bar{\chi})$.
Hence, problem~\eqref{eq:ssc_equiv_scalar_aux_minimization} has admissible points, and
we easily verify existence of a minimizer using the direct method.
Moreover, due to \cref{remark:assumption_non_degenerate} 
(or using the first order optimality condition $\partial_\nu g(\bar{\chi}) \neq 0$ and
linearity of $C_{\bar{q}}$), we have $\partial_{q}g(\bar{\chi}) C_{\bar{q}} = \R$,
which means that a constraint qualification condition (see, e.g.,~\cite{Zowe1979}) is fulfilled.
Thus, we obtain the necessary and
sufficient optimality conditions of the convex
problem~\eqref{eq:ssc_equiv_scalar_aux_minimization} in the
form~\eqref{eq:ssc_reduces_to_scalar_condition_system}.
Hence, for the positivity condition~\eqref{eq:secondOrderSufficientCondition} we only have
to require that \(\bar{\gamma} > 0\), which guarantees
\begin{equation}
\label{eq:mimimum_nu_lagrange}
\partial^2_{(\nu,q)} \mathcal{L}(\bar{\chi},\bar{\mu})[1,\delta{q}]^2
 \geq \partial^2_{(\nu,q)} \mathcal{L}(\bar{\chi},\bar{\mu})[1,\delta{\bar{q}}]^2
 = \bar{\gamma} > 0,
\end{equation}
for any \(\delta{q}\) with \((1,\delta{q}) \in \CriticalCone\),
where $\delta\bar{q}$ is the solution to~\eqref{eq:ssc_reduces_to_scalar_condition_system}.
\end{proof}

The system~\eqref{eq:ssc_reduces_to_scalar_condition_system} still involves the
solution of an infinite-dimensional linear-quadratic optimization
problem. However, the same calculation
is valid for the discrete problem, which can be used to numerically verify the SSC by computing the
constant \(\bar{\gamma}\) on the discrete level.
Note that while \(\bar{\gamma} > 0\) implies the SSC from
\cref{thm:secondOrderSufficient}, it does not represent a coercivity constant 
for the Hessian of the Lagrange function as in \cref{thm:secondOrderSufficientEquivalence}.
Instead, we can derive a lower bound on the coercivity constant in terms of
\(\bar{\gamma}\), which also depends explicitly on \(\alpha > 0\).
\begin{proposition}
\label{proposition:lower_bound_coercivity}
Let $(\bar{\nu},\bar{q}) \in \Rplus\times\Qad(0,1)$,
\cref{assumption:non_degenerate} hold, and assume
that \(\bar{\gamma} > 0\) (as defined in \cref{lemma:ssc_reduces_to_scalar_condition}).
Then, the coercivity constant from \cref{thm:secondOrderSufficientEquivalence} is bounded
from below by \(\bar{\kappa} \geq (\bar{\gamma}/3)
\min\set{\alpha\bar{\nu}/(\bar{\gamma}+c_1), 1}\), where \(c_1\) depends on the optimal solution.
\end{proposition}
\begin{proof}
By replacing \(\delta{q}\) with \(\delta{q}/\delta{\nu}\)
in~\eqref{eq:mimimum_nu_lagrange} and using linearity we directly obtain
\[
\partial^2_{(\nu,q)} \mathcal{L}(\bar{\chi},\bar{\mu})[\delta{\nu}, \delta{q}]^2 \geq \bar{\gamma}\abs{\delta{\nu}}^2
\quad\text{for all } (\delta{\nu},\delta{q}) \in \CriticalCone.
\]
Furthermore, by using the coercivity of
\(\partial^2_{q} \mathcal{L}(\bar{\chi},\bar{\mu})\) with constant \(\alpha\bar{\nu}\) and straightforward estimates
(using Young's inequality), we can derive that
\[
\partial^2_{(\nu,q)} \mathcal{L}(\bar{\chi},\bar{\mu})[\delta{\nu}, \delta{q}]^2 
\geq \frac{\alpha\bar{\nu}}{2} \norm{\delta{q}}^2_{L^2(I\times\omega)}
- c_1 \abs{\delta{\nu}}^2
\quad\text{for all } (\delta{\nu},\delta{q}),
\]
where \(c_1 = \left(\abs{\partial^2_{\nu}\mathcal{L}(\bar{\chi},\bar{\mu})}
  + 2\norm{\partial_{\nu}\partial_{q} \mathcal{L}(\bar{\chi},\bar{\mu})}^2/ (\alpha\bar{\nu})\right) \).
By taking a convex combination of \((1-\theta)\) times the former and \(\theta\) times the latter estimate, where
\(\theta = (2/3) (\bar{\gamma}/(\bar{\gamma} + c_1)) \), we arrive at
the desired estimate.
\end{proof}

\begin{remark}
We can also give an interpretation of \(\bar{\gamma}\) in terms of a certain
value function, which is introduced as
\begin{equation}
\label{eq:value_function}
V(\nu) = \;\min_{q \in \Qad(0,1),\; g(\nu,q) \leq 0}\; j(\nu, q) =
j(\nu, \bar{q}(\nu)).
% = \nu + \nu(\alpha/2)\norm{q}^2
\end{equation}
Thus, \(V\) is defined by fixing an arbitrary time \(\nu > 0\) and resolving the resulting
linear-quadratic optimization problem with optimal solution \(\bar{q}(\nu)\);
cf.\ also~\cite{Kunisch2013}.
Clearly, minimizing \(V\) delivers the optimal time \(\bar{\nu}\). Moreover, by
established perturbation arguments (cf., e.g., \cite{Griesse2007,Bonnans2000}) using
\cref{assumption:non_degenerate}, it can be shown that:
\begin{enumerate}[(i)]
\item
  \(V\) is finite in a neighborhood of \(\bar{\nu}\) and twice differentiable.
\item
  It holds \(V'(\nu) = \partial_\nu\mathcal{L}(\nu,\bar{q}(\nu),\bar{\mu}(\nu)) = 0\),
  where \(\bar{\mu}(\nu)\) is the multiplier for the minimization problem
  in~\eqref{eq:value_function}.
\item
  The derivative of \(\nu \mapsto (\bar{q}(\nu), \bar{\mu}(\nu))\) at \(\bar{\nu}\) is the
  unique solution of~\eqref{eq:ssc_reduces_to_scalar_condition_system}.
\end{enumerate}
Differentiating the expression in (ii) with respect to \(\nu\) and using (iii) together with
the concrete form of~\eqref{eq:ssc_reduces_to_scalar_condition_system} we obtain 
\begin{align*}
V''(\bar{\nu})
= \partial^2_{(\nu,q)} \mathcal{L}(\bar{\nu},\bar{q},\bar{\mu})[1,\delta\bar{q}]^2
= \bar{\gamma},
\end{align*}
since \((\bar{q}(\bar{\nu}), \bar{\mu}(\bar{\nu})) = (\bar{q},\bar{\mu})\).
Thus, the constant \(\bar{\gamma}\) can be interpreted as the local curvature of the value
function \(V\) around the optimal time.
\end{remark}

\section{Finite element discretization}\label{sec:discretization}
With the first and second order optimality conditions at hand, we can
now turn to the main subject of this paper, i.e.\ \emph{a priori} discretization error estimates for the time-optimal control problem~\eqref{P}.
First, we prove a suboptimal convergence result where we rely on the quadratic growth condition of \cref{thm:secondOrderSufficient}. 
Thereafter, we provide an optimal discretization error estimate for the control variable that is directly based on the second order sufficient optimality condition~\eqref{eq:secondOrderSufficientCondition}.
First of all, we discuss the discretization method and provide stability and discretization error estimates.

\subsection{Discretization and problem statement}
\label{subsec:discretization_assumptions_and_problem_statement}
Consider a partitioning of the (reference) time interval $[0, 1]$ given as
\begin{equation*}
[0, 1] = \{0\}\cup I_1\cup I_2 \cup\ldots \cup I_M
\end{equation*}
with disjoint subintervals $I_m = (t_{m-1},t_m]$ of size $k_m$ defined by the time points
\begin{equation*}
0 = t_0 < t_1 < \ldots < t_{M-1} < t_M = 1\mbox{.}
\end{equation*}
We abbreviate the time discretization by the parameter $k$ defined as the piecewise constant function by setting $k|_{I_m} = k_m$ for all $m = 1,2,\ldots,M$. Simultaneously, we denote by $k$ the maximal size of the time steps, i.e.\ $k = \max k_m$. 
Moreover, we assume that the regularity conditions for the time mesh from \cite[Section~3.1]{Meidner2011a} are satisfied.

Concerning the spatial discretization, we consider a discretization consisting of triangular or tetrahedral cells $K$ that constitute a non-overlapping cover of the domain $\Omega$. We define the discretization parameter $h$ as the cellwise constant function $h|_K = h_K$ with diameter $h_K$ of the cell $K$ and set $h = \max h_K$. The corresponding mesh is denoted by $\mathcal{T}_h = \{K\}$. 	
Let $V_h \subset H^1_0$ denote the subspace of cellwise linear and continuous functions.
Moreover, let $\ProjH\colon L^2 \rightarrow V_h$ be
the $L^2$-projection onto $V_h$.
We assume that $\ProjH$ is stable in $H^1$. This is satisfied if, e.g., the mesh is globally quasi-uniform but weaker conditions are known; cf.\ \cite{Bramble2002}.
The corresponding space-time finite element space is constructed in a standard way by
\begin{equation*}
\Xkh = \left\{v_{kh} \in L^2(I; V_h) \constraintSet v_{kh}|_{I_m} \in \mathcal P_0(I_m; V_h)\mbox{,~}m =1,2,\ldots,M \right\}\mbox{,}
\end{equation*}
where $\mathcal P_0(I_m; V_h)$ denotes the space of constant functions on the time interval~$I_m$ with values in~$V_h$.	
For any function $\varphi_k\in \Xkh$ we set $\varphi_{k,m} \ldef \varphi_k(t_m)$ with $m=1,2,\ldots,M,$ as well as
$[\varphi_k]_m \ldef \varphi_{k,m+1}-\varphi_{k,m}$ for $m=1,2,\ldots,M-1$.
Now, we define the trilinear form $\B \colon \R\times\Xkh\times\Xkh \rightarrow \R$ as
\begin{multline*}\label{eq:discreteFormB}
\B(\nu, u_{kh}, \varphi_{kh}) \ldef \sum_{m=1}^M \pair{\partial_t u_{kh}, \varphi_{kh}}_{L^2(I_m; L^2)}\\
 + \nu\inner{\nabla u_{kh}, \nabla\varphi_{kh}}_{L^2(I; L^2)} + \sum_{m=2}^M([u_{kh}]_{m-1}, \varphi_{kh,m}) + \inner{u_{kh,1},\varphi_{kh,1}}\mbox{.}
\end{multline*}
Note that the definition of $\B$ above can be directly extended on the larger space
$\Xkh + W(0,1)$, which allows to formulate Galerkin orthogonality.
Given $\nu \in \Rplus$ and $q \in Q(0,1)$ the discrete state equation reads as follows: Find a state $u_{kh} \in \Xkh$ satisfying
\begin{align}
\B(\nu, u_{kh},\varphi_{kh}) &= \nu\inner{\ControlOp q,\varphi_{kh}}_{L^2(I; L^2)} + \inner{u_0,\varphi_{kh,1}}_{L^2} &\text{for all }\varphi_{kh}\in \Xkh\mbox{.}\label{eq:stateEquationDiscrete}
\end{align}
To consider different control discretizations at the same time, we introduce the
operator $\ProjDiscControl$
onto the (possibly discrete) control space $\Qsigma(0,1) \subset \Lcontrol{2}$ with an abstract parameter $\sigma$ for the control discretization. 
In case of distributed control, 
we additionally assume that a subset denoted $\mathcal{T}^\omega_h$ of the mesh $\mathcal{T}_h$ is a non-overlapping cover of $\omega$. 
We use the symbol $\sigma(k,h)$ to denote the error due to control discretization, i.e.
\begin{equation}\label{eq:estimate_projection_discrete_controls}
\norm{q - \ProjDiscControl q}_{L^2(I\times\omega)}  \leq \sigma(k,h)\norm{q}_\sigma\mbox{,}
\end{equation}
where $\norm{\cdot}_\sigma$ stands for a potentially different norm of
a subspace of $\Q(0,1)$. We suppose $\sigma(k,h) \rightarrow 0$ as $k, h \rightarrow 0$
and $\ProjDiscControl\Qad(0,1) \subset \Qad(0,1)$. Moreover, we assume
$\norm{\bar{q}}_\sigma <\infty$ and $\norm{q}_{\Lcontrol{2}} \leq \norm{q}_\sigma$. For notational simplicity we write $\ProjDiscControl (\nu,q) = (\nu, \ProjDiscControl q)$ using the same symbol.
Concrete discretization strategies for the control will be discussed in \cref{sec:ImprovedErrorEstimates}.
For convenience we define $\Qsigmaad(0,1) = \Qsigma(0,1)\cap\Qad(0,1)$.

Analogous to the continuous solution operator and the reduced constraint mapping, for
\(\nu \in \Rplus\) and \(q\in L^2(I\times\omega)\) we introduce the discrete versions as
\(S_{kh}(\nu,q) \ldef u_{kh} \), where \(u_{kh}\) solves \eqref{eq:stateEquationDiscrete} and
\[
g_{kh}(\nu,q) \ldef G(i_1 S_{kh}(\nu,q))\mbox{.}
\]
The discrete optimal control problem now reads as follows:
\begin{equation}\label{Pkh}\tag{\mbox{$\hat{P}_{kh}$}}
\inf_{\substack{\nu_{kh}\in\Rplus\\ q_{kh} \in \Qsigmaad(0,1)}} j(\nu_{kh}, q_{kh})
\quad\mbox{subject to}\quad g_{kh}(\nu_{kh},q_{kh}) \leq 0\mbox{.}
\end{equation}
At this point, the well-posedness of~\eqref{Pkh} is not clear. In the
following, as a by-product of the error analysis, we will show existence of feasible
points for \(k\) and \(h\) small enough (using the linearized Slater
condition~\eqref{eq:definition_linearized_slater_condition}), which implies
existence for~\eqref{Pkh} by similar arguments as for the continuous problem.
Furthermore, we derive optimality conditions and rates of convergence of the
optimization variables, where the second order sufficient
condition~\eqref{eq:secondOrderSufficientCondition} is an essential ingredient for the latter.

\subsubsection{Stability estimates for the PDE}
We introduce the discrete analogue $-\Lap_h\colon V_h\to V_h$ to the operator $-\Lap$ as
\begin{equation*}
-\inner{\Lap_h u_h, \varphi_h}_{L^2} = \inner{\nabla u_h, \nabla \varphi_h}_{L^2}\mbox{,}\quad \varphi_h \in V_h\mbox{.}
\end{equation*}
For the discretization error estimates we require stability estimates for the state, linearized state, and adjoint state.
\begin{proposition}\label{prop:stabilityEstimateDiscreteState}
	For every tuple $(\nu, q) \in \Rplus\times\Q(0,1)$ there exists a unique solution $u_{kh} \in \Xkh$ to the discrete state equation. Moreover, there is $c > 0$ independent of $u_0, u_{kh}, \nu$, and $q$ such that the following stability estimates hold
	\begin{align}
	\norm{u_{kh}(1)}^2_{L^2} + \nu\norm{u_{kh}}_{L^2(I; H_0^1)}^2 &\leq c \left(\nu\norm{\ControlOp q}^2_{L^2(I; H^{-1})} + \norm{\ProjH u_0}^2_{L^2}\right)\mbox{,}\label{eq:stabilityEstimateDiscreteState1}\\
	\norm{\nabla u_{kh}(1)}_{L^2}^2 &\leq c \left(\nu\norm{\ControlOp q}^2_{L^2(I; L^2)} + \frac{1}{\nu}\norm{\ProjH u_0}^2_{L^2}\right)\mbox{.}\label{eq:stabilityEstimateDiscreteState2}
	\end{align}
\end{proposition}
\begin{proof}
	For the first estimate, we proceed as in \cite{Meidner2008a} and test with $\varphi = u_{kh}$.
	To show the second estimate, we consider first the case $u_0 = 0$ and test with $\varphi = -\Lap_h u_{kh}$, and thereafter the case $q = 0$ where we test with $\varphi = -\nu t_m \Lap_h u_{kh}$ as in the proof of \cite[Theorem~4.5]{Meidner2011a}. Superposition of both estimates yields~\eqref{eq:stabilityEstimateDiscreteState2}.		
\end{proof}

\begin{corollary}\label{corollary:stabilityEstimateDiscreteTangent}
	Let $u_{kh}\in\Xkh$ be the state corresponding to $(\nu, q) \in \Rplus\times\Q(0,1)$. For all $(\delta\nu, \delta q) \in \R\times\Q(0,1)$ there are unique solutions $\delta u_{kh} \in \Xkh$ and $\delta \tilde{u}_{kh} \in \Xkh$ to the discrete linearized and second linearized state equation, i.e.
	\begin{align*}
	\B(\nu, \delta u_{kh},\varphi_{kh}) &= \inner{\delta\nu (\ControlOp q + \Lap_h u_{kh}) + \nu\ControlOp\delta q,\varphi_{kh}}_{L^2(I; L^2)}\mbox{,}\\
	\B(\nu, \delta \tilde{u}_{kh},\varphi_{kh}) &= 2\inner{\delta\nu (\ControlOp \delta q + \Lap_h\delta u_{kh}) ,\varphi_{kh}}_{L^2(I; L^2)}\mbox{,}
	\end{align*}
	for all $\varphi_{kh} \in \Xkh$.
	Moreover, it holds
	\begin{align*}
	\norm{\delta u_{kh}(1)}^2_{L^2}
	&\leq c \left(\abs{\delta\nu}^2(\norm{\ControlOp q}^2_{L^2(I; H^{-1})} + \frac{1}{\nu}\norm{\ProjH u_0}^2_{L^2}) + \nu\norm{\ControlOp\delta q}^2_{L^2(I; L^2)}\right)\mbox{,}\\
	\norm{\delta \tilde{u}_{kh}(1)}^2_{L^2}
	&\leq c \abs{\delta\nu}^2\left(\norm{\ControlOp q}^2_{L^2(I; H^{-1})} + \norm{\delta u_{kh}}_{L^2(I; H^1)}^2\right)\mbox{.}
	\end{align*}
	The constant $c > 0$ is independent of $u_0, \nu, q, \delta\nu, \delta q, \delta u_{kh}$, and $\delta\tilde{u}_{kh}$.
\end{corollary}
Similarly, we obtain for the auxiliary adjoint equation the following stability result.
\begin{proposition}\label{prop:stabilityEstimateDiscreteAdjoint}
	For every triple $(\nu, f, z_1) \in \Rplus\times L^2(I; L^2)\times H^1_0$ there exists a unique solution $\tilde{z}_{kh} \in \Xkh$ to 
	\begin{equation*}
	\B(\nu,\varphi_{kh},\tilde{z}_{kh}) = \nu\inner{f,\varphi_{kh}}_{L^2(I; L^2)} + \inner{z_1,\varphi_{kh}(1)}\mbox{,}\quad \varphi_{kh} \in \Xkh\mbox{.}
	\end{equation*}
	Moreover, there is $c > 0$ independent of $\tilde{z}_{kh}, \nu$, and $f$ such that
	\begin{align*}
	\norm{\tilde{z}_{kh}}_{L^2(I; H_0^1)} &\leq c\left(\norm{f}_{L^2(I; L^2)} + \frac{1}{\sqrt{\nu}}\norm{\ProjH z_1}_{L^2}\right)\mbox{,}\\
	\norm{\Lap_h \tilde{z}_{kh}}_{L^2(I; L^2)}&\leq c\left(\norm{f}_{L^2(I; L^2)} + \frac{1}{\sqrt{\nu}}\norm{\ProjH z_1}_{H^1}\right)\mbox{.}
	\end{align*}
\end{proposition}	
As in the continuous case we obtain a discrete analogue to \cref{prop:stability_g} using the stability estimates of \cref{prop:stabilityEstimateDiscreteState} and \cref{corollary:stabilityEstimateDiscreteTangent} for the discrete states. 
\begin{proposition}\label{prop:stabilityEstimateGkh}
	Let $0 < \nu_{\min} < \nu_{\max}$ be given. Then there exists $c > 0$ independent of $k$ and $h$ such that
	for all $\delta\nu \in \R$ and $\delta q \in \Lcontrol{2}$ it holds
	\begin{align}
	\abs{g_{kh}'(\nu, q)(\delta\nu, \delta q)} & \leq c\prodnorm{\delta\nu, \delta q}\label{eq:stabilityEstimateGkh1}\mbox{,}\\
	\abs{g_{kh}''(\nu, q)[\delta\nu, \delta q]^2}  &\leq c \prodnorm{\delta\nu, \delta q}^2\label{eq:stabilityEstimateGkh2}\mbox{.}
	\end{align}
	for all $\nu_{\min} \leq \nu \leq \nu_{\max}$ and $q \in \Qad(0,1)$.
	Moreover, $g_{kh}$ and $g_{kh}'$ are Lipschitz continuous on bounded sets.
\end{proposition}

\subsubsection{Discretization error for terminal constraint}
Next, we establish discretization error estimates concerning the reduced constraint
function $g$. Note that general error estimates for the state equation are collected in
\cref{appendix:discretization_error_estimates}.

\begin{proposition}\label{prop:errorEstimateDiscreteAdjoint}
	Let $0 < \nu_{\min} < \nu_{\max}$, $(\nu, q) \in [\nu_{\min}, \nu_{\max}]\times\Qad(0,1)$, and $\mu \in \R$. For the adjoint state $z$ defined in~\eqref{eq:adjoint_state_equation} associated with $u = u(\nu,q)$ and the discrete adjoint state $z_{kh}$ defined by
	\begin{equation*}
	\B(\nu,\varphi_{kh},z_{kh}) = \mu\inner{u_{kh}(1)-u_d,\varphi_{kh}(1)}\mbox{,}\quad \varphi_{kh} \in \Xkh\mbox{,}
	\end{equation*}
	associated with $u_{kh} = u_{kh}(\nu,q)$ it holds
	\begin{align}
	\norm{z - z_{kh}}_{L^2(I; L^2)}&\leq c\abs{\log k}(k + h^2) \left(\norm{\ControlOp q}_{L^\infty(I; L^2)} + \norm{u_0}_{L^2}\right) \abs{\mu},\label{eq:discretizationErrorEstimateAdjointL2L2}\\
	\norm{\nabla z - \nabla z_{kh}}_{L^2(I; L^2)}&\leq c\abs{\log k}(k^{1/2} + h) \left(\norm{\ControlOp q}_{L^\infty(I; L^2)} + \norm{u_0}_{L^2}\right) \abs{\mu},\label{eq:discretizationErrorEstimateAdjointL2H1}
	\end{align}
	where $c > 0$ is a constant independent of $z, z_{kh}, \nu$, and $q$.
\end{proposition}
\begin{proof}
	We consider the splitting
	\begin{equation}
	\label{eq:discretizationErrorEstimateAdjointP1}
	z - z_{kh} = z - \tilde{z} + \tilde{z} - z_{kh}\mbox{,}
	\end{equation}
	where $\tilde{z}$ denotes the solution to
	\begin{equation*}
	-\partial_t \tilde{z} - \nu \Lap\tilde{z} = 0\mbox{,}\quad \tilde{z}(1) = \mu(u_{kh}(1)-u_d)\mbox{.}
	\end{equation*}
	By means of the stability estimates \cref{prop:stabilityS} for $u$ and \cref{prop:stabilityEstimateDiscreteState} for $u_{kh}$ as well as boundedness of $q \in \Qad(0,1)$ and $\nu \in [\nu_{\min}, \nu_{\max}]$ we find that $u(1)$ and $u_{kh}(1)$ are uniformly bounded in $L^2$. Employing a stability result similar as \cref{prop:stabilityS} and Lipschitz continuity of $G'$ on bounded sets in $L^2$ we infer
	\begin{align}
	\norm{z - \tilde{z}}_{L^2(I; H^1)} &\leq c\frac{\abs{\mu}}{\sqrt{\nu}}\norm{u(1) - u_{kh}(1)}_{L^2}\nonumber\\
	&\leq c(\nu_{\min}, \nu_{\max})\abs{\mu}\abs{\log k}(k + h^2)\left(\norm{\ControlOp q}_{L^\infty(I; L^2)} + \norm{u_0}_{L^2}\right)\label{eq:discretizationErrorEstimateAdjointP3},
	\end{align}
	where we have used the discretization error estimate~\eqref{eq:errorEstimatesStateEquationLinfL2} in the last step. The second term in~\eqref{eq:discretizationErrorEstimateAdjointP1} is a pure discretization error, therefore,
	\begin{align*}
	\norm{\tilde{z} - z_{kh}}_{L^2(I; L^2)} &\leq c(k + h^2) \abs{\mu} \norm{u_{kh}(1)}_{H^1},\\
	\norm{\nabla\tilde{z} - \nabla z_{kh}}_{L^2(I; L^2)} &\leq c(k^{1/2} + h)  \abs{\mu}\norm{u_{kh}(1)}_{H^1}\mbox{;}
	\end{align*}
	cf.\ \eqref{eq:errorEstimatesStateEquationL2L2} and~\eqref{eq:errorEstimatesStateEquationL2H1}.
	The assertion follows from~\eqref{eq:discretizationErrorEstimateAdjointP3}, the two preceding estimates and the stability estimates~\eqref{eq:stabilityEstimateDiscreteState1} and~\eqref{eq:stabilityEstimateDiscreteState2} applied for $u_{kh}$.
\end{proof}

\begin{proposition}\label{prop:errorEstimatesG}
	Let $0 < \nu_{\min} < \nu_{\max}$ be fixed. Consider $(\nu, q) \in [\nu_{\min}, \nu_{\max}]\times\Qad(0,1)$ and $(\delta\nu, \delta q) \in \R\times\Q(0,1)$. Then there is $c > 0$ independent of $(\nu, q)$ and $(\delta\nu, \delta q)$ such that
	\begin{align}
	\abs{g(\nu,q) - g_{kh}(\nu,q)} &\leq c\abs{\log k} (k + h^2)\left(\norm{\ControlOp q}_{L^\infty(I; L^2)} + \norm{u_0}_{L^2}\right)\mbox{,}\label{eq:errorEstimatesG1}\\
	\abs{(g'(\nu,q) - g_{kh}'(\nu,q))(\delta\nu, \delta q)} &\leq c \abs{\log k} (k + h^2)\left(\norm{\ControlOp q}_{L^\infty(I; L^2)} + \norm{u_0}_{H^1}\right)\prodnorm{\delta\nu, \delta q}\mbox{,}\label{eq:errorEstimatesG2}
	\end{align}
	where $c > 0$ is a constant independent of $\nu, q, \delta\nu$, and $\delta q$.
	%		 (but depending on $\nu_{\min}, \nu_{\max}$, and $\Qad$).
\end{proposition}
\begin{proof}
	From the discretization error estimate~\eqref{eq:errorEstimatesStateEquationLinfL2}
	and Lipschitz continuity of $G$ on bounded sets in $L^2$ we conclude
	\begin{equation*}
	\abs{g(\nu,q) - g_{kh}(\nu,q)} \leq c(\nu) \norm{u(1) - u_{kh}(1)}_{L^2} \leq c(\nu)\abs{\log k} (k + h^2)\mbox{.}
	\end{equation*}		
	To prove~\eqref{eq:errorEstimatesG2}, we use the adjoint representation~\eqref{eq:terminal_constraint_adjoint} and its discrete analogue. Let $\mu \in \R$, then
	\begin{equation*}
	\mu [g'(\nu,q) - g_{kh}'(\nu,q)]^*
	%		&= \left(\begin{array}{l}
	%		\pair{q + \Lap u, z} - \pair{q + \Lap u_{kh}, z_{kh}}\\
	%		\nu z - \nu z_{kh}
	%		\end{array}
	%		\right)\\
	= \left(\begin{array}{l}
	\int_{0}^{1}\pair{\ControlOp q, z - z_{kh}} + \pair{\Lap u, z} - \pair{\Lap_h u_{kh}, z_{kh}}\D{t}\\
	\nu \ControlOp^*(z -z_{kh})
	\end{array}
	\right).
	\end{equation*}
	Clearly, the terms involving $z - z_{kh}$  can be estimated using~\eqref{eq:discretizationErrorEstimateAdjointL2L2}. Concerning the remaining terms of the first component, we have
	\begin{equation*}
	\pair{\Lap u, z} - \pair{\Lap_h u_{kh}, z_{kh}} = 
	-\pair{u_{kh} - u, \Lap z} + \pair{\nabla u_{kh}-\nabla u, \nabla z_{kh}-\nabla z} -  \pair{\Lap  u, z_{kh}-z}\mbox{.}
	\end{equation*}
	Since $\Lap u, \Lap z \in L^2(I; L^2)$, we conclude
	\begin{align*}
	\abs{\pair{\Lap_h u_{kh}, z_{kh}} - \pair{\Lap u, z}} 
	&\leq c(\norm{u_{kh} - u}_{L^2(I; L^2)}\abs{\mu} + \norm{z_{kh} - z}_{L^2(I; L^2)}\\
	&\quad+ \norm{\nabla u_{kh} - \nabla u}_{L^2(I; L^2)}\norm{\nabla z_{kh} - \nabla z}_{L^2(I; L^2)})\\
	&\leq c(\nu)\abs{\log k} (k + h^2)\abs{\mu}\left(\norm{\ControlOp q}_{L^\infty(I; L^2)} + \norm{u_0}_{H^1}\right)
	\end{align*}
	according to~\eqref{eq:errorEstimatesStateEquationL2L2},~\eqref{eq:discretizationErrorEstimateAdjointL2L2},~\eqref{eq:errorEstimatesStateEquationL2H1}, and~\eqref{eq:discretizationErrorEstimateAdjointL2H1}.
	Thus, we obtain~\eqref{eq:errorEstimatesG2}.
\end{proof}

\subsection{Convergence analysis}\label{subsec:convergenceAnalysis}
In order to deal with local solutions, we apply a standard localization argument, cf.\ \cite{Casas2002c}. For a given locally optimal control $(\bar{\nu}, \bar{q})$ of~\eqref{Pt} in $\Qad\cap\overline{\mathcal{B}_\rho(\bar{\nu},\bar{q})}$ with $\rho > 0$ sufficiently small
satisfying the linearized Slater condition \cref{assumption:linearized_slater}, 
we introduce the auxiliary problem
\begin{equation}\label{PkhLocal}\tag{\mbox{$\hat{P}_{kh}^\rho$}}
\inf_{\substack{\nu_{kh}\in\Rplus\\ q_{kh} \in \Qsigmaad(0,1)}} j(\nu_{kh}, q_{kh})
\quad\mbox{subject to}\quad
\begin{cases}
 &g_{kh}(\nu_{kh},q_{kh}) \leq 0\mbox{,} \\
 &\prodnorm{\nu_{kh} - \bar{\nu}, q_{kh} - \bar{q}} \leq \rho\mbox{.}
\end{cases}
\end{equation}
We first construct a sequence of tuples $\{(\nu_\gamma,q_\gamma)\}_{\gamma > 0}$
converging to $(\bar{\nu},\bar{q})$ as $\gamma\to 0$ that is feasible for the localized
problem (for sufficiently small \(k\) and \(h\)).
In particular, this implies existence of solutions to~\eqref{PkhLocal}. Thereafter we
construct a sequence $\{(\nu_\tau,q_\tau)\}_{\tau > 0}$ converging to
$(\bar{\nu}_{kh}^\rho,\bar{q}_{kh}^\rho)$ as $\tau\to 0$ that is feasible for~\eqref{Pt}.
Feasibility of the $\tau$-sequence for~\eqref{Pt} with the quadratic growth
condition~\eqref{eq:quadraticGrowthCondition} yields convergence of discrete solutions to
$(\bar{\nu},\bar{q})$ at a suboptimal rate. The convergence result will later be the basis
for the improved convergence rate in \cref{sec:ImprovedErrorEstimates}.

In order to ensure that the constants in the following arguments are independent of
\(\bar{\nu}^\rho_{kh}\), we have to guarantee that $\bar{\nu}^\rho_{kh}$ is uniformly
bounded away from zero;
cf., e.g., \cref{prop:stabilityEstimateGkh,prop:errorEstimatesG,appendix:discretization_error_estimates}.
To this end, we always assume in the following that \(\rho \leq \bar{\nu}/2\), which implies
\(\bar{\nu}/2 \leq \bar{\nu}^\rho_{kh} \leq (3/2)\bar{\nu}\) by the localization in~\eqref{PkhLocal}.

\subsubsection{The localized discrete problem}

In the following, we will repeatedly make use of the Slater point
\(\breve{\chi}^\gamma\) defined in~\eqref{eq:definition_linearized_slater_condition_general}.
We start by constructing admissible elements for the discrete problems.
\begin{proposition}\label{prop:auxiliarySequenceGamma}
	Let $(\bar{\nu},\bar{q})$ be a locally optimal control of
        problem~\eqref{Pt}. There exists a sequence $\{(\nu_\gamma,q_\gamma)\}_{\gamma > 0}$ of controls with $\gamma = \gamma(k,h)$ that are feasible
        for~\eqref{PkhLocal} for $k,h$ sufficiently small. Moreover,
	\begin{equation*}
	\abs{\nu_{\gamma}-\bar{\nu}} + \norm{q_{\gamma}-\bar{q}}_{L^2(I\times\omega)}
	\leq c\left(\sigma(k,h) + \abs{\log k} (k + h^2)\right)\mbox{.}
	\end{equation*}
\end{proposition}
\begin{proof}
	The proof follows the one of \cite[Lemma~4.2]{Neitzel2015}. 
	We abbreviate $\bar{\chi} = (\bar{\nu},\bar{q})$. Moreover, for $\gamma > 0$
    to be determined in the course of the proof we set
	\begin{equation*}
	\chi_\gamma \ldef \ProjDiscControl \breve{\chi}^\gamma
        = (\bar{\nu}+\gamma,\ProjDiscControl \bar{q}) \mbox{.}
	\end{equation*}
	Employing the supposition~\eqref{eq:estimate_projection_discrete_controls} on
        $\ProjDiscControl$ we obtain
	\begin{equation}\label{eq:auxiliarySequenceGammaP1}
	\norm{\chi_\gamma - \bar{\chi}} \leq \gamma + \sigma(k,h)\norm{\bar{q}}_\sigma\mbox{.}
	\end{equation}
	Moreover, using Taylor expansion of $g_{kh}$ at $\ProjDiscControl\bar{\chi}$ we
        find for some \(\chi_\zeta\) that
	\begin{align*}
	g_{kh}(\chi_\gamma)
          &= g_{kh}(\ProjDiscControl\bar{\chi})
            + \gamma \, g_{kh}'(\ProjDiscControl\bar{\chi})(1,0)
	+ \frac{\gamma^2}{2} \, g_{kh}''(\chi_\zeta)[1,0]^2\mbox{.}
	\end{align*}
        Using the triangle inequality we estimate the first term by
	\begin{align*}
	g_{kh}(\ProjDiscControl\bar{\chi})
        &\leq g(\bar{\chi}) + \abs{g(\bar{\chi}) - g_{kh}(\bar{\chi})}
        + c\norm{\ProjDiscControl\bar{\chi} - \bar{\chi}} \\
        &\leq c_1(\abs{\log k} (k + h^2) + \sigma(k,h))
        \rdef \delta_1(k,h)
	\end{align*}
        with Lipschitz continuity of $g_{kh}$ and \cref{prop:errorEstimatesG}.
        For the second term, we estimate
        \[
          g_{kh}'(\ProjDiscControl\bar{\chi})(1,0)
          \leq g'(\bar{\chi})(1,0) + c_2\left(\abs{\log k} (k + h^2) + \sigma(k,h)\right)
          \leq -\bar{\eta} + \delta_2(k,h)\mbox{,}
        \]
        using \cref{assumption:linearized_slater}, and \(g'(\bar{\chi})(1,0)
        = \partial_\nu g(\bar{\chi})\).
	Finally, for the third term, we find
	\(g_{kh}''(\chi_\zeta)[\gamma,0]^2 \leq c_3 \gamma^2\), using~\eqref{eq:stabilityEstimateGkh2}.	
	Collecting the estimates, we have
	\begin{equation*}
	g_{kh}(\chi_{\gamma}) \leq \delta_1(k,h)
        - \gamma\left(\bar{\eta} - \delta_2(k,h) - c_3 \gamma\right)\mbox{.}
	\end{equation*}
	Note that the first component of $\chi_\gamma$ is bounded below by
        $\bar{\nu}$ and bounded above by $\bar{\nu}+1$, so that all constants of
        \cref{prop:stabilityEstimateGkh,prop:errorEstimatesG} can be chosen to be
        independent of $\chi_\gamma$. Taking 
	\begin{equation*}
          \gamma = \frac{3 \delta_1(k,h)}{\bar{\eta}} \leq \frac{\bar{\eta}}{3c_3}
          \quad\text{and}\quad
          \delta_2(k,h) \leq \frac{\bar{\eta}}{3}
	\end{equation*}
        for $k, h$ sufficiently small, we obtain $g_{kh}(\chi_\gamma) \leq 0$. From the 
        definition of $\gamma$ we further deduce $\gamma = \gamma(k,h) =
        \mathcal{O}(\sigma(k,h) + \abs{\log k} (k + h^2))$. Moreover, it holds
        $\norm{\chi_\gamma - \bar{\chi}} \leq \rho$ for $\gamma,k,h$ sufficiently small
        due to~\eqref{eq:auxiliarySequenceGammaP1}. In summary, we have that the sequence
        $\chi_{\gamma}$ is feasible for~\eqref{PkhLocal}.
\end{proof}
In particular, \cref{prop:auxiliarySequenceGamma} guarantees that for $h$,  $k$, and
$\rho$ sufficiently small, the set of admissible controls of the discrete
problem~\eqref{PkhLocal} is nonempty. 
Hence, by standard arguments we obtain well-posedeness of the localized discrete problem.
\begin{corollary}
	Let $h$,  $k$, and $\rho$ be sufficiently small. 
	Then there exists a solution $\bar{\chi}_{kh}^\rho = (\bar{\nu}^\rho_{kh},\bar{q}^\rho_{kh}) \in \Rplus\times\Qsigmaad(0,1)$ 
	to~\eqref{PkhLocal}.
\end{corollary}
First, we verify that the linearized Slater condition holds at $\bar{\chi}^\rho_{kh}$ for the discrete problem.
\begin{proposition}\label{prop:slaterPointDiscrete}
	For $k$, $h$, and $\rho$ sufficiently small we have
	\begin{equation*}
	\partial_\nu g_{kh}(\bar{\chi}^\rho_{kh}) \leq -\bar{\eta}/2 < 0\mbox{.}
	\end{equation*}
\end{proposition}
\begin{proof}
  This follows with \cref{assumption:linearized_slater} and
  \begin{equation*}
    \partial_\nu g_{kh}(\bar{\chi}^\rho_{kh})
    \leq \partial_\nu g(\bar{\chi}) + \abs{\partial_\nu g_{kh}(\bar{\chi}^\rho_{kh})
      - \partial_\nu g(\bar{\chi}^\rho_{kh})} + \abs{\partial_\nu g(\bar{\chi}^\rho_{kh})
      - \partial_\nu g(\bar{\chi})}\mbox{,}
  \end{equation*}
  using the error estimate \eqref{eq:errorEstimatesG2}, the
  Lipschitz-continuity of \(\partial_\nu g_{kh}\) from \eqref{eq:stabilityEstimateGkh1},
  and the fact that \(\norm{\bar{\chi}^\rho_{kh} - \bar{\chi}} \leq \rho\) by the construction
  of~\eqref{PkhLocal}.
\end{proof}

Last, we construct a sequence that is feasible for~\eqref{Pt}
and its distance to $(\bar{\nu}^\rho_{kh},\bar{q}^\rho_{kh})$ converges at the rate $\abs{\log k} (k + h^2)$.
\begin{proposition}\label{prop:auxiliarySequenceTau}
	Let $k$, $h$, and $\rho$ be sufficiently small. Moreover, let
        $(\bar{\nu},\bar{q})$ be a locally optimal solution of~\eqref{Pt} and let
        $(\bar{\nu}^\rho_{kh},\bar{q}^\rho_{kh})$ be any globally optimal control
        of~\eqref{PkhLocal}. Then there exists a sequence $\{\nu_\tau\}_{\tau >
          0}$ with $\tau = \tau(k,h)$ 
        such that $(\nu_\tau,\bar{q}^\rho_{kh})$
        is feasible for~\eqref{Pt} and
        that fulfill
	\begin{equation*}
	\abs{\nu_{\tau}-\bar{\nu}^\rho_{kh}}
	\leq c\abs{\log k} (k + h^2)\mbox{.}
	\end{equation*}
\end{proposition}
\begin{proof}
We set
\[
  \chi_\tau = (\nu_\tau, q_\tau) = (\bar{\nu}^\rho_{kh} + \tau, \bar{q}^\rho_{kh})
\]
for some \(\tau \in (0,1]\) to be determined later.
Now, the proof proceeds along the lines of the proof of \cref{prop:auxiliarySequenceGamma},
interchanging the roles of \(\bar{\chi}\) and \(\bar{\chi}_{kh}\) and \(g\) and \(g_{kh}\).
\end{proof}

\subsubsection{A priori error estimates of controls}
Two-way insertion of the auxiliary sequences constructed in the preceding subsection,
combined with the quadratic growth condition, yields a first convergence result.
\begin{proposition}\label{prop:errorEstimatesControlLocal}
	Let $(\bar{\nu},\bar{q})$ be a local solution to~\eqref{Pt}. Moreover, let $\{(k,h)\}$ be a sequence of positive mesh sizes converging to zero and $\{(\bar{\nu}^\rho_{kh},\bar{q}^\rho_{kh})\}_{k,h>0}$ be a sequence of globally optimal solutions to~\eqref{PkhLocal} for $\rho > 0$ sufficiently small such that the quadratic growth condition~\eqref{eq:quadraticGrowthCondition} as well as \cref{prop:auxiliarySequenceTau,prop:auxiliarySequenceGamma} hold. Then $(\bar{\nu}^\rho_{kh},\bar{q}^\rho_{kh})$ converges to $(\bar{\nu},\bar{q})$ and
	\begin{equation*}
	\abs{\bar{\nu} - \bar{\nu}_{kh}^\rho} + \norm{\bar{q} - \bar{q}_{kh}^\rho}_{L^2(I\times\omega)}
	\leq c\left(\sigma(k,h)^{1/2} + \abs{\log k}^{1/2} (k^{1/2} + h)\right)\mbox{.}
	\end{equation*}
\end{proposition}
\begin{proof}
	Because $(\nu_\tau, \bar{q}_{kh}^\rho)$ is feasible for~\eqref{Pt},
	we may use the quadratic growth condition~\eqref{eq:quadraticGrowthCondition}
	to estimate
	\begin{multline*}
	\frac{\kappa}{2} \prodnorm{\bar{\nu} - \nu_\tau, \bar{q} - \bar{q}_{kh}^\rho}^2 \leq j(\nu_\tau, \bar{q}_{kh}^\rho) - j(\bar{\nu},\bar{q})\\
	\leq j(\nu_\tau, \bar{q}_{kh}^\rho) - j(\bar{\nu}_{kh}^\rho, \bar{q}_{kh}^\rho) +  j(\bar{\nu}_{kh}^\rho, \bar{q}_{kh}^\rho) - j(\nu_\gamma,q_\gamma) + j(\nu_\gamma,q_\gamma) - j(\bar{\nu},\bar{q})\\
	\leq j(\nu_\tau, \bar{q}_{kh}^\rho) - j(\bar{\nu}_{kh}^\rho, \bar{q}_{kh}^\rho) + j(\nu_\gamma,q_\gamma) - j(\bar{\nu},\bar{q})\mbox{,}
	\end{multline*}
	where the last inequality follows from optimality of the pair $(\bar{\nu}_{kh}^\rho, \bar{q}_{kh}^\rho)$ and feasibility of $(\nu_\gamma,q_\gamma)$ for~\eqref{PkhLocal}.		
	Then, we observe
	\begin{align*}
	j(\nu_\tau, \bar{q}_{kh}^\rho) - j(\bar{\nu}_{kh}^\rho, \bar{q}_{kh}^\rho) &=
	(\nu_{\tau}-\bar{\nu}_{kh}^\rho)\left(1 + \frac{\alpha}{2}\norm{\bar{q}_{kh}^\rho}_{\Lcontrol{2}}^2\right)\\
	&\leq c\left(1 + \frac{\alpha}{2}\right)\abs{\log k} (k + h^2)
	\end{align*}
	due to \cref{prop:auxiliarySequenceTau} and boundedness of $\bar{q}_{kh}^\rho$. Similarly, 
	\begin{align*}
	j(\nu_\gamma, q_\gamma) - j(\bar{\nu}, \bar{q}) &=
	(\nu_{\gamma}-\bar{\nu})\left(1 + \frac{\alpha}{2}\norm{q_\gamma}_{\Lcontrol{2}}^2\right)\\
	&\quad + \bar{\nu}\frac{\alpha}{2}\norm{q_\gamma + \bar{q}}_{\Lcontrol{2}} \norm{q_\gamma - \bar{q}}_{\Lcontrol{2}}\\
	&\leq c\left(1 + \frac{\alpha}{2}\right)(\sigma(k,h) + \abs{\log k} (k + h^2))
	\end{align*}
	employing \cref{prop:auxiliarySequenceGamma}. Taking square roots yields the assertion.
\end{proof}

\begin{lemma}\label{lemma:errorEstimatesControl_suboptimal}
	Let $(\bar{\nu},\bar{q})$ be a local solution to~\eqref{Pt} satisfying the quadratic growth condition~\eqref{eq:quadraticGrowthCondition} and $\{(k,h)\}$ be a sequence of positive mesh sizes converging to zero. There is a sequence $\{(\bar{\nu}_{kh},\bar{q}_{kh})\}_{k,h>0}$ of local solutions to problem~\eqref{Pkh} such that
	\begin{equation*}
	\abs{\bar{\nu} - \bar{\nu}_{kh}} + \norm{\bar{q} - \bar{q}_{kh}}_{L^2(I\times\omega)} \leq c\left(\sigma(k,h)^{1/2} + \abs{\log k}^{1/2} (k^{1/2} + h)\right)\mbox{,}
	\end{equation*}
	where $c > 0$ is independent of $k$, $h$, $\bar{\nu}_{kh}$, and $\bar{q}_{kh}$.
	Moreover, there exists a Lagrange multiplier $\bar{\mu}_{kh} > 0$ such that the
        following optimality system is satisfied:
	\begin{align}
	%\bar{\mu}_{kh} &> 0\mbox{,}\label{eq:opt_cond_multiplier_positive_kh}\\
	\int_0^1 1 + \frac{\alpha}{2}\norm{\bar{q}_{kh}}_{\LcontrolSpatial}^2 + \pair{\ControlOp\bar{q}_{kh} + \Lap_h\bar{u}_{kh}, \bar{z}_{kh}}\D{t} &= 0\mbox{,}\label{eq:opt_cond_hamiltonianConstant_kh}\\
	\int_0^{1} \bar{\nu}_{kh}\pair{\alpha \bar{q}_{kh}+\ControlOp^*\bar{z}_{kh}, q - \bar{q}_{kh}} \D{t} &\geq 0\mbox{,}\quad q \in\Qsigmaad(0,1)\mbox{,}\label{eq:opt_cond_variationalInequality_kh}\\
	G(\bar{u}_{kh}(1)) &= 0\mbox{,}\label{eq:opt_cond_feasibility_kh}
	\end{align}
	where $\bar{u}_{kh} = S_{kh}(\bar{\nu}_{kh},\bar{q}_{kh})$ and $\bar{z}_{kh} \in \Xkh$ is the discrete adjoint equation, i.e.
	\begin{equation*}
	\B(\bar{\nu}_{kh}, \varphi_{kh}, \bar{z}_{kh}) = \bar{\mu}_{kh}\inner{\bar{u}_{kh}(1) - u_d, \varphi_{kh}(1)}\mbox{,}\quad \varphi_{kh} \in \Xkh\mbox{.}
	\end{equation*}
\end{lemma}
\begin{proof}
	The assertion follows from \cref{prop:errorEstimatesControlLocal}, noting that
        global solutions of~\eqref{PkhLocal} are local solutions of~\eqref{Pkh}, since the
        constraint $\prodnorm{\nu_{kh} - \bar{\nu}, q_{kh} - \bar{q}} \leq \rho$ is not
        active for sufficiently small $k$ and $h$, due to the convergence result of \cref{prop:errorEstimatesControlLocal}.
	Furthermore, \cref{prop:slaterPointDiscrete} guarantees the existence of KKT multipliers satisfying the optimality system stated above.
\end{proof}

\begin{proposition}\label{prop:convergenceMultipliers}
	Adopt the assumptions of \cref{lemma:errorEstimatesControl_suboptimal}. Then it holds
	\begin{equation}\label{eq:errorestimate_multiplier}
	\abs{\bar{\mu}-\bar{\mu}_{kh}} \leq c\left(\abs{\log k}(k + h^2) + \prodnorm{\bar{\nu} - \bar{\nu}_{kh}, \bar{q} - \bar{q}_{kh}}\right),
	\end{equation}
	with a constant $c > 0$ independent of $k$, $h$, $\bar{\nu}_{kh}$, $\bar{q}_{kh}$, and $\bar{\mu}_{kh}$.
\end{proposition}
\begin{proof}
We abbreviate $\bar{\chi} = (\bar{\nu},\bar{q})$ and $\bar{\chi}_{kh} = (\bar{\nu}_{kh},\bar{q}_{kh})$.	
Combining the optimality conditions for~\eqref{Pt} and~\eqref{Pkh} we obtain
\[
  \bar{\mu}-\bar{\mu}_{kh} = \partial_\nu g(\bar{\chi})^{-1}\partial_\nu j(\bar{\chi}) 
  - \partial_\nu g_{kh}(\bar{\chi}_{kh})^{-1}\partial_\nu j(\bar{\chi}_{kh}).
\]
Now, we may use the discretization error estimate~\eqref{eq:errorEstimatesG2} to infer
\begin{align*}
  \abs{\bar{\mu}-\bar{\mu}_{kh}} &\leq 
                                   \abs{\partial_\nu g(\bar{\chi})^{-1} - \partial_\nu g_{kh}(\bar{\chi})^{-1}}\partial_\nu j(\bar{\chi})\\
	&\quad + \abs{\partial_\nu g_{kh}(\bar{\chi})^{-1}\partial_\nu j(\bar{\chi}) - \partial_\nu g_{kh}(\bar{\chi}_{kh})^{-1}\partial_\nu j(\bar{\chi}_{kh})}\\			
	&\leq \frac{\abs{\partial_\nu g(\bar{\chi})-\partial_\nu g_{kh}(\bar{\chi})}}{\abs{\partial_\nu g(\bar{\chi})\partial_\nu g_{kh}(\bar{\chi})}}\partial_\nu j(\bar{\chi})
	+ \frac{\abs{\partial_\nu g_{kh}(\bar{\chi})-\partial_\nu g_{kh}(\bar{\chi}_{kh})}}{\abs{\partial_\nu g_{kh}(\bar{\chi})\partial_\nu g_{kh}(\bar{\chi}_{kh})}}\partial_\nu j(\bar{\chi})\\
	&\quad + \abs{\partial_\nu g_{kh}(\bar{\chi}_{kh})^{-1}}\abs{\partial_\nu j(\bar{\chi}) - \partial_\nu j(\bar{\chi}_{kh})}\\
	&\leq c\abs{\log k} (k + h^2) + c\norm{\bar{\chi}-\bar{\chi}_{kh}},
	\end{align*}
	where we have used that $\partial_\nu j(\chi) = \int_0^1(1+(\alpha/2)\norm{q}^2)$ and that
        $\abs{\partial_\nu g_{kh}(\bar{\chi})}
        \geq \bar{\eta}/2$ for \(k\) and \(h\) small enough, using again the discretization
        error estimate~\eqref{eq:errorEstimatesG2}.
\end{proof}

\subsection{Optimal error estimates of controls}\label{sec:ImprovedErrorEstimates}

Using the convergence result of the preceding subsection, we now prove optimal order of convergence with respect to the control variable. While the previous result is based on the quadratic growth condition, we now directly rely on the second order sufficient optimality condition and thus avoid taking square roots in the end. The improved convergence result will be consequence of the following Lemma.

%% Approach: Improved error estimates based on SSC
%%
%%
\begin{lemma}\label{lemma:ImprovedErrorRate}
	Let $(\bar{\nu},\bar{q})$ be a local solution to~\eqref{Pt} satisfying the second
        order sufficient optimality condition~\eqref{eq:secondOrderSufficientCondition}
        and let $\{(k,h)\}$ be a sequence of positive mesh sizes such that $\abs{\log k}(k
        + h^2) \to 0$. Let $\{(\bar{\nu}_{kh},\bar{q}_{kh})\}_{k,h>0}$ be a sequence of
        local solutions to~\eqref{Pkh} converging in $\R\times\Lcontrol{2}$ and associated
        Lagrange multipliers $\bar{\mu}_{kh}$ converging in $\R$. Then there are constants
        $c > 0$ and $k_0, h_0 > 0$ such that
	\begin{equation}\label{eq:improvedErrorEstimate}
	\prodnorm{\bar{\nu} - \bar{\nu}_{kh}, \bar{q} - \bar{q}_{kh}}^2
	\leq c \left[\abs{\log k}^2 (k + h^2)^2
          + \norm{\bar{q}-q_{kh}}_{\Lcontrol{2}}^2
          + \partial_q\mathcal{L}(\bar{\nu},\bar{q},\bar{\mu})(q_{kh}-\bar{q})\right]
	\end{equation}
	for all $q_{kh} \in \Qsigmaad(0,1)$ and all $k \leq k_0$ and $h \leq h_0$.
\end{lemma}
\begin{proof}
	We adapt the ideas of the proof of Theorem~2.14 in \cite{Casas2012d} for optimal
        control problems without state constraints. Instead of working with the objective
        functional, we use the Lagrange function $\mathcal{L}$ and the corresponding
        second order sufficient optimality
        condition~\eqref{eq:secondOrderSufficientCondition}. We abbreviate $\bar{\chi} =
        (\bar{\nu},\bar{q})$ and $\bar{\chi}_{kh} = (\bar{\nu}_{kh},\bar{q}_{kh})$.
	
	\textit{Step 0: Preparation.} Since $(\bar{\nu},\bar{q})$ is optimal for~\eqref{Pt}, it holds
	\begin{equation}\label{eq:improvedErrorEstimateP1}
	\partial_{\chi}\mathcal{L}(\bar{\chi},\bar{\mu})(\chi-\bar{\chi})\geq 0
	\end{equation}
	for all $\chi \in \Rplus \times\Qad(0,1)$, and by the same arguments for the discrete problem~\eqref{Pkh}
	\begin{equation}\label{eq:improvedErrorEstimateP2}
	\partial_{\chi}\mathcal{L}_{kh}(\bar{\chi}_{kh},\bar{\mu}_{kh})(\chi_{kh}-\bar{\chi}_{kh})\geq 0
	\end{equation}
	for all $\chi_{kh} \in \Rplus \times\Qsigmaad(0,1)$.		
	Using~\eqref{eq:improvedErrorEstimateP1} and $\Qsigma(0,1) \subset \Q(0,1)$, we find
	\begin{multline}
	\label{eq:improvedErrorEstimateP3}
	\partial_{\chi}\left[\mathcal{L}(\bar{\chi}_{kh},\bar{\mu})-\mathcal{L}(\bar{\chi},\bar{\mu})\right](\bar{\chi}_{kh}-\bar{\chi})
	\leq \partial_{\chi}\mathcal{L}(\bar{\chi}_{kh},\bar{\mu})(\bar{\chi}_{kh}-\bar{\chi})\\
	\leq \partial_{\chi}\left[\mathcal{L}(\bar{\chi}_{kh},\bar{\mu})-\mathcal{L}(\bar{\chi}_{kh},\bar{\mu}_{kh})\right](\bar{\chi}_{kh}-\bar{\chi}) + \partial_{\chi}\mathcal{L}(\bar{\chi}_{kh},\bar{\mu}_{kh})(\bar{\chi}_{kh}-\bar{\chi})\mbox{.}
	\end{multline}
	The first term on the right-hand side of~\eqref{eq:improvedErrorEstimateP3} satisfies
	\begin{equation*}
	\partial_{\chi}\left[\mathcal{L}(\bar{\chi}_{kh},\bar{\mu})-\mathcal{L}(\bar{\chi}_{kh},\bar{\mu}_{kh})\right](\bar{\chi}_{kh}-\bar{\chi}) = (\bar{\mu}-\bar{\mu}_{kh})g'(\bar{\chi}_{kh})(\bar{\chi}_{kh}-\bar{\chi})\mbox{.}
	\end{equation*}
	Concerning the second term on the right-hand side of~\eqref{eq:improvedErrorEstimateP3}, using~\eqref{eq:improvedErrorEstimateP2} and inserting additional terms with some arbitrary $\chi_{kh} \in \Rplus\times\Qsigmaad(0,1)$ yield
	\begin{align}
	\partial_{\chi}\mathcal{L}(\bar{\chi}_{kh}&,\bar{\mu}_{kh})(\bar{\chi}_{kh}-\bar{\chi})
	\leq \partial_{\chi}\mathcal{L}(\bar{\chi}_{kh},\bar{\mu}_{kh})(\bar{\chi}_{kh}-\bar{\chi}) + \partial_{\chi}\mathcal{L}_{kh}(\bar{\chi}_{kh},\bar{\mu}_{kh})(\chi_{kh}-\bar{\chi}_{kh})\nonumber\\
	&= \partial_{\chi}\left[\mathcal{L}_{kh}(\bar{\chi}_{kh},\bar{\mu}_{kh})-\mathcal{L}(\bar{\chi}_{kh},\bar{\mu}_{kh})\right](\bar{\chi}-\bar{\chi}_{kh}) + \partial_{\chi}\mathcal{L}_{kh}(\bar{\chi}_{kh},\bar{\mu}_{kh})(\chi_{kh}-\bar{\chi}) \nonumber\\
	&= \partial_{\chi}\left[\mathcal{L}_{kh}(\bar{\chi}_{kh},\bar{\mu}_{kh})-\mathcal{L}(\bar{\chi}_{kh},\bar{\mu}_{kh})\right](\bar{\chi}-\bar{\chi}_{kh})\nonumber\\
	&\quad + \partial_{\chi}\left[\mathcal{L}_{kh}(\bar{\chi}_{kh},\bar{\mu}_{kh})-\mathcal{L}(\bar{\chi}_{kh},\bar{\mu}_{kh})\right](\chi_{kh}-\bar{\chi})\nonumber\\
	&\quad+ \partial_{\chi}\left[\mathcal{L}(\bar{\chi}_{kh},\bar{\mu}_{kh})-\mathcal{L}(\bar{\chi},\bar{\mu}_{kh})\right](\chi_{kh}-\bar{\chi}) + \partial_{\chi}\mathcal{L}(\bar{\chi},\bar{\mu}_{kh})(\chi_{kh}-\bar{\chi})\mbox{.}\label{eq:improvedErrorEstimateP5}
	\end{align}
	Concerning the first term on the right-hand side, we find
	\begin{align*}
	\partial_{\chi}\left[\mathcal{L}_{kh}(\bar{\chi}_{kh},\bar{\mu}_{kh})-\mathcal{L}(\bar{\chi}_{kh},\bar{\mu}_{kh})\right](\bar{\chi}-\bar{\chi}_{kh}) &= \bar{\mu}_{kh}\left[g'_{kh}(\bar{\chi}_{kh})-g'(\bar{\chi}_{kh})\right](\bar{\chi}-\bar{\chi}_{kh})\\
	&\leq c \abs{\log k} (k + h^2) \norm{\bar{\chi}-\bar{\chi}_{kh}}\mbox{,}
	\end{align*}
	where we have used boundedness of the Lagrange multipliers $\bar{\mu}_{kh}$ due to \cref{prop:convergenceMultipliers} and the estimate~\eqref{eq:errorEstimatesG2}. Similarly for the second term of~\eqref{eq:improvedErrorEstimateP5}, it holds
	\begin{equation*}
	\partial_{\chi}\left[\mathcal{L}_{kh}(\bar{\chi}_{kh},\bar{\mu}_{kh})-\mathcal{L}(\bar{\chi}_{kh},\bar{\mu}_{kh})\right](\chi_{kh}-\bar{\chi}) \leq c \abs{\log k} (k + h^2) \norm{\chi_{kh}-\bar{\chi}}\mbox{.}
	\end{equation*}
	The third term of~\eqref{eq:improvedErrorEstimateP5} is estimated using Lipschitz continuity of $\partial_{\chi}\mathcal{L}$ (due to Lipschitz continuity of $g'$ on bounded sets)
	\begin{equation*}
	\partial_{\chi}\left[\mathcal{L}(\bar{\chi}_{kh},\bar{\mu}_{kh})-\mathcal{L}(\bar{\chi},\bar{\mu}_{kh})\right](\chi_{kh}-\bar{\chi}) \leq c \norm{\bar{\chi}_{kh}-\bar{\chi}}\norm{\chi_{kh}-\bar{\chi}}\mbox{.}
	\end{equation*}
	Since $\mathcal{L}$ is two times continuously differentiable we find
	\begin{equation}\label{eq:improvedErrorEstimateP8}
	\partial_{\chi}^2\mathcal{L}(\check{\chi}_{kh},\bar{\mu})[\bar{\chi}_{kh}-\bar{\chi}]^2 = \partial_{\chi}\left[\mathcal{L}(\bar{\chi}_{kh},\bar{\mu})-\mathcal{L}(\bar{\chi},\bar{\mu})\right](\bar{\chi}_{kh}-\bar{\chi})
	\end{equation}
	with $\check{\chi}_{kh}$ in between $\bar{\chi}$ and $\bar{\chi}_{kh}$. Together with the estimates above, we obtain
	\begin{equation}\label{eq:improvedErrorEstimateP9}
	\begin{aligned}
	\partial_{\chi}^2\mathcal{L}(\check{\chi}_{kh},\bar{\mu})[\bar{\chi}_{kh}-\bar{\chi}]^2 &\leq c \abs{\log k} (k + h^2) \left(\norm{\bar{\chi}-\bar{\chi}_{kh}} + \norm{\bar{\chi}-\chi_{kh}}\right)\\	
	&\quad+ c\norm{\bar{\chi}_{kh}-\bar{\chi}}\norm{\chi_{kh}-\bar{\chi}} 
	+ \partial_{\chi}\mathcal{L}(\bar{\chi},\bar{\mu}_{kh})(\chi_{kh}-\bar{\chi})\\
	&\quad+ \abs{\bar{\mu}-\bar{\mu}_{kh}}\abs{g'(\bar{\chi}_{kh})(\bar{\chi}_{kh}-\bar{\chi})}\mbox{.}
	\end{aligned}
	\end{equation}
	
	We argue by contradiction. Suppose that~\eqref{eq:improvedErrorEstimate} is false, then there exist a subsequence of mesh sizes $\{k_n,h_n\}$ converging to zero and $(\bar{\nu}_n,\bar{q}_n) \in \Rplus\times\Qsigmaad(0,1)$ such that $(\bar{\nu}_n,\bar{q}_n) \to (\bar{\nu},\bar{q})$ with
	\begin{equation*}
	\norm{\bar{\chi}_{n}-\bar{\chi}}^2 > n\left[(\abs{\log k_n} (k_n + h_n^2))^2 + \norm{q_n-\bar{q}}_{\Lcontrol{2}}^2 + \partial_{q}\mathcal{L}(\bar{\chi},\bar{\mu})( q_n-\bar{q})\right]\mbox{,}
	\end{equation*}
	where we use for convenience the short notation $\bar{\nu}_{n} = \bar{\nu}_{k_n h_n}$ and $\mathcal{L}_{n} = \mathcal{L}_{k_n h_n}$ etc. Setting $\chi_n = (\bar{\nu}, q_n)$, the inequality is equivalent to 
	\begin{equation}\label{eq:improvedErrorEstimateP10}
	\frac{1}{n} > \frac{(\abs{\log k_n} (k_n + h_n^2))^2}{\norm{\bar{\chi}_n-\bar{\chi}}^2} + \frac{\norm{\chi_n-\bar{\chi}}^2}{\norm{\bar{\chi}_n-\bar{\chi}}^2} + \frac{\partial_{\chi}\mathcal{L}(\bar{\chi},\bar{\mu})(\chi_n-\bar{\chi})}{\norm{\bar{\chi}_n-\bar{\chi}}^2}\mbox{.}
	\end{equation}
	Define $\rho_n = \norm{\bar{\chi}_n-\bar{\chi}}$ and
	\begin{equation*}
	v_n = (v^\nu_n, v^q_n) =\frac{1}{\rho_n}(\bar{\chi}_n-\bar{\chi})\mbox{.}
	\end{equation*}
	We may assume w.l.o.g.\ that  $v^\nu_n \rightarrow v^\nu$ in $\R$ and $v_n^q \rightharpoonup v^q$ in $\Lcontrol{2}$ and we abbreviate $v = (v^\nu, v^q)$.
	
	\textit{Step 1: $\partial_{\chi}\mathcal{L}(\bar{\chi},\bar{\mu})v = 0$. }
	The optimality condition~\eqref{eq:optimalityCondLagrange} implies
	\begin{equation*}
	\partial_{\chi}\mathcal{L}(\bar{\chi},\bar{\mu})v 
	= \lim_{n\rightarrow\infty} \partial_{\chi}\mathcal{L}(\bar{\chi},\bar{\mu})v_n \geq 0\mbox{.}
	\end{equation*}
	To show the reverse inequality, we consider
	\begin{align}
	\partial_{\chi}\mathcal{L}(\bar{\chi},\bar{\mu})v &= \lim_{n\rightarrow\infty} \partial_{\chi}\mathcal{L}(\bar{\chi},\bar{\mu})v_n\label{eq:improvedErrorEstimateP11}\\
	&= \lim_{n\rightarrow\infty}\partial_{\chi}\mathcal{L}_n(\bar{\chi}_n,\bar{\mu}_n)v_n\nonumber\\
	&\quad + \lim_{n\rightarrow\infty}\partial_{\chi}\left[\mathcal{L}(\bar{\chi}_n,\bar{\mu}_n) - \mathcal{L}_n(\bar{\chi}_n,\bar{\mu}_n)\right]v_n\nonumber\\
	&\quad+ \lim_{n\rightarrow\infty}\partial_{\chi}\left[\mathcal{L}(\bar{\chi},\bar{\mu}) - \mathcal{L}(\bar{\chi}_n,\bar{\mu}_n)\right]v_n\mbox{.}\label{eq:improvedErrorEstimateP12}
	\end{align}
	The limit in~\eqref{eq:improvedErrorEstimateP11} exists due to weak convergence of $(v^\nu_n, v^q_n)$. Concerning the second limit in~\eqref{eq:improvedErrorEstimateP12} we observe
	\begin{multline*}
	\lim_{n\rightarrow\infty}\left[\partial_{\chi}\mathcal{L}(\bar{\chi}_n,\bar{\mu}_n) - \partial_{\chi}\mathcal{L}_n(\bar{\chi}_n,\bar{\mu}_n)\right]v_n\\
	= \lim_{n\rightarrow\infty}\bar{\mu}_n\left[g'(\bar{\chi}_n) - g'_n(\bar{\chi}_n)\right]v_n
	\leq c \lim_{n\rightarrow\infty}\abs{\log k_n} (k_n + h_n^2) = 0\mbox{,}
	\end{multline*}
	where we have used boundedness of $\bar{\mu}_n$ and~\eqref{eq:errorEstimatesG2}. Using Lipschitz continuity we estimate the third limit as
	\begin{equation*}
	\lim_{n\rightarrow\infty}\left[\partial_{\chi}\mathcal{L}(\bar{\chi},\bar{\mu}) - \partial_{\chi}\mathcal{L}(\bar{\chi}_n,\bar{\mu}_n)\right]v_n
	\leq c \lim_{n\rightarrow\infty}\left(\norm{\bar{\chi}_n-\bar{\chi}} + \abs{\bar{\mu}-\bar{\mu}_{n}}\right) = 0\mbox{,}
	\end{equation*}
	due to $\norm{v_n} = 1$ and convergence of $\bar{\mu}_{n}$; see \cref{prop:convergenceMultipliers}. Thus, the first limit in~\eqref{eq:improvedErrorEstimateP12} must exist as well. 	
	Using continuity of $\partial_{\chi}\mathcal{L}$ in $\R\times\Lcontrol{2}$ and the optimality condition~\eqref{eq:improvedErrorEstimateP2} for $\bar{\chi}_n = (\bar{\nu}_n,\bar{q}_n)$ with $\chi_n = (\bar{\nu}, q_n)$ we find
	\begin{align*}
	\partial_{\chi}\mathcal{L}(\bar{\chi},\bar{\mu})v
	&\leq \lim_{n\rightarrow\infty}\partial_{\chi}\mathcal{L}_n(\bar{\chi}_n,\bar{\mu}_n)v_n\\
	&= \lim_{n\rightarrow\infty}\frac{1}{\rho_n}\left[\partial_{\chi}\mathcal{L}_n(\bar{\chi}_n,\bar{\mu}_n)(0,q_n-\bar{q})+\partial_{\chi}\mathcal{L}_n(\bar{\chi}_n,\bar{\mu}_n)(\bar{\nu}_n-\bar{\nu},\bar{q}_n-q_n)\right]\\
	&\leq \lim_{n\rightarrow\infty}\frac{1}{\rho_n}\partial_{\chi}\mathcal{L}_n(\bar{\chi}_n,\bar{\mu}_n)(0,q_n-\bar{q})\mbox{.}
	\end{align*}
	Since for any $\varphi \in \R\times\Lcontrol{2}$ it holds
	\begin{align*}
	\partial_{\chi}\mathcal{L}_n(\bar{\chi}_n,\bar{\mu}_n)\varphi
	&\leq |\partial_{\chi}\mathcal{L}(\bar{\chi}_n,\bar{\mu}_n)\varphi| + \left|\left[\partial_{\chi}\mathcal{L}_n(\bar{\chi}_n,\bar{\mu}_n) -\partial_{\chi}\mathcal{L}(\bar{\chi}_n,\bar{\mu}_n)\right]\varphi\right|\\
	&\leq c\left(1 + \abs{\log k_n} (k_n + h_n^2)\right) \prodnorm{\varphi^\nu,\varphi^q}\mbox{,}
	\end{align*}
	we conclude
	\begin{equation*}
	\partial_{\chi}\mathcal{L}(\bar{\chi},\bar{\mu})v
	\leq\lim_{n\rightarrow\infty}c\left(1 + \abs{\log k_n} (k_n + h_n^2)\right)\frac{\norm{q_n-\bar{q}}_{\Lcontrol{2}}}{\rho_n} = 0\mbox{,}
	\end{equation*}
	due to~\eqref{eq:improvedErrorEstimateP10}. In summary, we proved $\partial_{\chi}\mathcal{L}(\bar{\chi},\bar{\mu})v = 0$.
	
	\textit{Step 2: $g'(\bar{\chi})v = 0$. }
	Using $g(\bar{\chi}) = g_n(\bar{\chi}_n) = 0$,~\eqref{eq:errorEstimatesG1},~\eqref{eq:improvedErrorEstimateP10}, and step 1 we infer
	\begin{align*}
	j'(\bar{\chi})v &= \lim_{n\rightarrow\infty}\frac{1}{\rho_n}\left[j(\bar{\chi}_n)-j(\bar{\chi})\right] 
	= \lim_{n\rightarrow\infty}\frac{1}{\rho_n}\left[\mathcal{L}_n(\bar{\chi}_n,\bar{\mu})-\mathcal{L}(\bar{\chi},\bar{\mu})\right]\\
	&= \lim_{n\rightarrow\infty}\frac{1}{\rho_n}\left[\mathcal{L}_n(\bar{\chi}_n,\bar{\mu})-\mathcal{L}(\bar{\chi}_n,\bar{\mu})+\mathcal{L}(\bar{\chi}_n,\bar{\mu})-\mathcal{L}(\bar{\chi},\bar{\mu})\right]\\
	&\leq \limsup_{n\rightarrow\infty}\frac{c}{\rho_n}\abs{\log k_n} (k_n + h_n^2) + \partial_{\chi}\mathcal{L}(\bar{\chi},\bar{\mu})v = 0\mbox{.}
	\end{align*}
	Similarly, we calculate
	\begin{align*}
	g'(\bar{\chi})v &= \lim_{n\rightarrow\infty} \frac{1}{\rho_n} \left[g(\bar{\chi}_n) - g(\bar{\chi})\right]
	=\lim_{n\rightarrow\infty} \frac{1}{\rho_n} \left[(g_{n}(\bar{\chi}_n) - g(\bar{\chi})) + (g(\bar{\chi}_n) - g_{n}(\bar{\chi}_n))\right]\\
	&\leq \limsup_{n\rightarrow \infty}\frac{c}{
		\rho_n}\abs{\log k_n} (k_n + h_n^2) = 0\mbox{.}
	\end{align*}
	Hence, from $\partial_{\chi}\mathcal{L}(\bar{\chi},\bar{\mu})v = j'(\bar{\chi})v + \bar{\mu}\,g'(\bar{\chi})v = 0$
	and $\bar{\mu} > 0$ (see \cref{lemma:first_order_optcond}), we conclude $g'(\bar{\chi})v = 0$.
	
	\textit{Step 3: $v \in \CriticalCone$. }
	Because the set
	\begin{equation*}
	\left\{\delta{q} \in  L^2(I\times\omega) \;\left|\; 
	\begin{aligned}
	\delta{q} \leq 0 &\mbox{~if~} \bar{q}(t,x) = q_b \\
	\delta{q} \geq 0 &\mbox{~if~} \bar{q}(t,x) = q_a
	\end{aligned}\right.
	\right\}\mbox{,}
	\end{equation*}
	is closed and convex, it is in particular weakly closed. Moreover, due to feasibility of $q_n$ every $(q_n-\bar{q})/\rho_n$ belongs to the set above, so does the weak limit. Thus, $v$ satisfies $v^q \leq 0$, if $\bar{q}(t,x) = q_b$, and $v^q \geq 0$, if $\bar{q}(t,x) = q_a$. For this reason,~\eqref{eq:optimalitySignConditionControlPlusAdjoint} implies
	\[
	\int_0^1\int_{\omega}\bar{\nu}(\alpha\bar{q}+\ControlOp^*\bar{z})v^q\D{x}\D{t}
	= \int_0^1\int_{\omega}\bar{\nu}\abs{(\alpha\bar{q}+\ControlOp^*\bar{z})v^q}\D{x}\D{t}.
	\]
	Moreover, due to $\partial_\chi\mathcal{L}(\bar{\chi},\bar{\mu})v = 0$ and the first order necessary condition $ \partial_\nu\mathcal{L}(\bar{\chi},\bar{\mu}) = 0$ we have the equality
	\begin{equation*}
	0 = \partial_q\mathcal{L}(\bar{\chi},\bar{\mu})v^q = \int_0^1\bar{\nu}\inner{\alpha\bar{q}+\ControlOp^*\bar{z},v^q}_{L^2(\omega)}\D{t} = \int_0^1\int_{\omega}\bar{\nu}\abs{(\alpha\bar{q}+\ControlOp^*\bar{z})v^q}\D{x}\D{t}\mbox{.}
	\end{equation*}		
	Hence, $v^q = 0$, if $\alpha\bar{q}(t,x)+\ControlOp^*\bar{z}(t,x) \neq 0$, and $v^q$ satisfies the sign condition~\eqref{eq:secondOrderSignCondition} as well. With step 1 we have proved that $v \in \CriticalCone$.		
	
	\textit{Step 4: $v = 0$.} Since $\bar{\chi}_n\rightarrow\bar{\chi}$ in $\R\times\Lcontrol{2}$, it holds $\check{\chi}_n\rightarrow\bar{\chi}$, where $\check{\chi}_n$ was defined in~\eqref{eq:improvedErrorEstimateP8}. Thus, continuity of $\partial_{\chi}\mathcal{L}$ in $\R\times\Lcontrol{2}$ yields
	\begin{align}
	\liminf_{n\rightarrow\infty} \partial_{\chi}^2\mathcal{L}(\check{\chi}_n,\bar{\mu})v_n^2
	&\geq \liminf_{n\rightarrow\infty}\partial_{\chi}^2\mathcal{L}(\bar{\chi},\bar{\mu})v_n^2 +
	\liminf_{n\rightarrow\infty} \partial_{\chi}^2[\mathcal{L}(\check{\chi}_n,\bar{\mu}) -\mathcal{L}(\bar{\chi},\bar{\mu})]v_n^2 \nonumber\\
	& = \liminf_{n\rightarrow\infty}\partial_{\chi}^2\mathcal{L}(\bar{\chi},\bar{\mu})v_n^2\mbox{.}\label{eq:improvedErrorEstimateP13}
	\end{align}		
	Due to~\eqref{eq:errorestimate_multiplier} and~\eqref{eq:improvedErrorEstimateP10} we have
	\begin{multline}
	\frac{1}{\rho_n^2}\partial_{\chi}\left[\mathcal{L}(\bar{\chi},\bar{\mu}_{n})-\mathcal{L}(\bar{\chi},\bar{\mu})\right](\chi_{n}-\bar{\chi}) 
	= \frac{1}{\rho_n^2}(\bar{\mu}_n-\bar{\mu})g'(\bar{\chi})(\chi_{n}-\bar{\chi})\\
	\leq c\frac{\abs{\bar{\mu}-\bar{\mu}_{n}}}{\norm{\bar{\chi}_n-\bar{\chi}}} \frac{\norm{\chi_{n}-\bar{\chi}}}{\norm{\bar{\chi}_n-\bar{\chi}}}
	\leq \frac{c}{\sqrt{n}}\left(\frac{\abs{\log k_n} (k_n + h_n^2)}{\norm{\bar{\chi}_n-\bar{\chi}}} + 1\right) \leq \frac{c}{\sqrt{n}}.
	\label{eq:improvedErrorEstimateP14}
	\end{multline}
	Similarly, using \eqref{eq:errorestimate_multiplier} and since $\abs{g'(\bar{\chi})v_n} \to 0$ by step 2, it holds
	\begin{align}
	\frac{\abs{\bar{\mu}-\bar{\mu}_{n}}\abs{g'(\bar{\chi}_n)(\bar{\chi}_{n}-\bar{\chi})}}{\norm{\bar{\chi}_n-\bar{\chi}}^2}
	&\leq \frac{\abs{\bar{\mu}-\bar{\mu}_{n}}}{\norm{\bar{\chi}_n-\bar{\chi}}}\left(\abs{g'(\bar{\chi})v_n} + \abs{[g'(\bar{\chi}_n)-g'(\bar{\chi})]v_n}\right)\nonumber\\
	&\leq c\left(\frac{\abs{\log k_n} (k_n + h_n^2)}{\norm{\bar{\chi}_n-\bar{\chi}}} + 1\right)
	\left(\abs{g'(\bar{\chi})v_n} + \norm{\bar{\chi}_n-\bar{\chi}}\right) \to 0.	
	\label{eq:improvedErrorEstimateP15}
	\end{align}
	Employing~\eqref{eq:improvedErrorEstimateP13} and~\eqref{eq:improvedErrorEstimateP9} we infer
	\begin{align}
	\liminf_{n\rightarrow\infty}\partial_{\chi}^2\mathcal{L}(\bar{\chi},\bar{\mu})v_n^2 &\leq \liminf_{n\rightarrow\infty} \partial_{\chi}^2\mathcal{L}(\check{\chi}_n, \bar{\mu})v_n^2
	\leq \limsup_{n\rightarrow\infty} \partial_{\chi}^2\mathcal{L}(\check{\chi}_n, \bar{\mu})v_n^2\nonumber\\
	&\leq \limsup_{n\rightarrow\infty} \left(\frac{c \abs{\log k_n} (k_n + h_n^2)}{\norm{\bar{\chi}_n-\bar{\chi}}} \left(1+ \frac{\norm{\chi_{n}-\bar{\chi}}}{\norm{\bar{\chi}_n-\bar{\chi}}}\right) + c\frac{\norm{\chi_{n}-\bar{\chi}}}{\norm{\bar{\chi}_n-\bar{\chi}}}\right.\nonumber\\
	&\quad +\frac{\partial_{\chi}\mathcal{L}(\bar{\chi},\bar{\mu})(\chi_{n}-\bar{\chi})}{\norm{\bar{\chi}_n-\bar{\chi}}^2} + \frac{\partial_{\chi}\left[\mathcal{L}(\bar{\chi},\bar{\mu}_{n})-\mathcal{L}(\bar{\chi},\bar{\mu})\right](\chi_{n}-\bar{\chi})}{\norm{\bar{\chi}_n-\bar{\chi}}^2}\nonumber\\
	&\quad +\left.\frac{\abs{\bar{\mu}-\bar{\mu}_{n}}\abs{g'(\bar{\chi}_n)(\bar{\chi}_{n}-\bar{\chi})}}{\norm{\bar{\chi}_n-\bar{\chi}}^2}\right) = 0\mbox{.}\label{eq:improvedErrorEstimateP20}
	\end{align}
	Here, we have used~\eqref{eq:improvedErrorEstimateP10} to estimate the first three summands,~\eqref{eq:improvedErrorEstimateP14} for the second last term, and~\eqref{eq:improvedErrorEstimateP15} for the last term.
	Last, weak lower semicontinuity of $j''$ and $g''$, and \cref{corollary:weak_lower_semicontinuity_g} lead to
	\begin{equation*}
	\partial_{\chi}^2\mathcal{L}(\bar{\chi},\bar{\mu})v^2 \leq 
	\liminf_{n\rightarrow\infty}\partial_{\chi}^2\mathcal{L}(\bar{\chi},\bar{\mu})v_n^2 \leq 0\mbox{.}
	\end{equation*}
	From the second order sufficient condition~\eqref{eq:secondOrderSufficientCondition} we conclude $v = (v^\nu,v^q) = 0$. Note that this in particular implies $v^\nu \to 0$ in $\R$.
	
	\textit{Step 5: Final contradiction. } 
	Using $\prodnorm{v_n^\nu, v_n^q} = 1$ and $v^\nu \to 0$ we obtain
	\begin{equation*}
	0 < \alpha \bar{\nu} = \alpha \bar{\nu} \liminf_{n\rightarrow\infty}\prodnorm{v_n^\nu, v_n^q}^2
	= \liminf_{n\rightarrow\infty}\alpha\int_0^1 \bar{\nu}\norm{v^q_n(t)}_{\LcontrolSpatial}^2\D{t}\mbox{.}
	\end{equation*}
	Using the specific structure of $j''$, we see that
	\begin{equation*}
	\liminf_{n\rightarrow\infty}\alpha\int_0^1 \bar{\nu}\norm{v^q_n(t)}_{\LcontrolSpatial}^2\D{t} = \liminf_{n\rightarrow\infty} j''(\bar{\chi})[v^\nu_n, v^q_n]^2\mbox{.}
	\end{equation*}
	Due to $g''(\bar{\chi})[0,0]^2 = 0$ and weak lower semicontinuity, see~\cref{corollary:weak_lower_semicontinuity_g}, we conclude
	\begin{align*}
	0 &< \liminf_{n\rightarrow\infty} j''(\bar{\chi})v_n^2
	\leq \liminf_{n\rightarrow\infty} j''(\bar{\chi})v_n^2 + \bar{\mu}\liminf_{n\rightarrow\infty} g''(\bar{\chi})v_n^2\\
	&\leq \liminf_{n\rightarrow\infty} \partial_{\chi}^2\mathcal{L}(\bar{\chi},\bar{\mu})v_n^2 \leq 0\mbox{,}
	\end{align*}
	where we have used again~\eqref{eq:improvedErrorEstimateP20} in the last inequality.
\end{proof}

Finally we prove the main result of this paper, i.e.\ \emph{a priori} discretization error estimates that are optimal with respect to the control variable. We consider different control discretization strategies.

\subsubsection{Parameter control and variational discretization}
As proposed in \cite{Hinze2005} for elliptic equations, cf.\ also \cite{Meidner2008} for parabolic equations, the state and adjoint equations are discretized, only. The control is then implicitly discretized employing the optimality conditions, precisely the discrete analogue to~\eqref{eq:projectionFormulaControl}. In this case, the operator $\ProjDiscControl$ is the identity and $\sigma(k,h) = 0$.

\begin{theorem}[Variational discretization]
	\label{thm:optimalconvergence_variational}
	Let the assumptions of \cref{lemma:errorEstimatesControl_suboptimal} hold and suppose the variational control discretization, i.e.\ $\Qsigma(0,1) = \Q(0,1)$. Then there is a constant $c > 0$ not depending on $h$ and $k$ such that
	\begin{equation*}
	%		\prodnorm{\bar{\nu} - \bar{\nu}_{kh}, \bar{q} - \bar{q}_{kh}} 
	\abs{\bar{\nu} - \bar{\nu}_{kh}} + \norm{\bar{q} - \bar{q}_{kh}}_{L^2(I\times\omega)} \leq c \abs{\log k} (k + h^2)\mbox{.}
	\end{equation*}
\end{theorem}
\begin{proof}
	\Cref{lemma:errorEstimatesControl_suboptimal} guarantees the existence of a sequence of local solutions converging strongly in $\R\times\Lcontrol{2}$. Hence, we can apply \cref{lemma:ImprovedErrorRate} with $q_{kh} = \bar{q}_{kh}$.
\end{proof}
In case of purely time-dependent control, the set $\omega$ is already discrete
and the space $L^2(\omega) \cong \R^{N_c}$ does not need to be discretized; see \cref{sec:notation_assumptions}.
Moreover, in view of the projection formula
\[
\bar{q}_{kh} = \ProjQad{-\frac{1}{\alpha}\ControlOp^*\bar{z}_{kh}},
\]
which can be deduced from~\eqref{eq:opt_cond_variationalInequality_kh} with
$\Qsigmaad(0,1) = \Qad(0,1)$, the optimal control $\bar{q}_{kh}$ obtained by the
variational approach is piecewise constant in time with values in $\R^{N_c}$.
Based on this observation, the controls constructed in
\cref{thm:optimalconvergence_variational} are already contained in a discrete space, and we
obtain the following corollary.
\begin{corollary}[Parameter control]
\label{cor:optimalconvergence_parameter}
	Let the assumptions of \cref{lemma:errorEstimatesControl_suboptimal} hold, suppose that
        \(\omega\) is discrete, and choose the piecewise constant discrete control space
        \(\Qsigma(0,1) = \left\{v \in \Q(0,1) \constraintSet  v|_{I_m} \in
            \mathcal{P}_0({I_m};\R^{N_c}), \; m = 1,2,\ldots,M\right\}\).
        Then there is a constant $c > 0$ not depending on $h$ and $k$ such that
	\begin{equation*}
	\abs{\bar{\nu} - \bar{\nu}_{kh}} + \norm{\bar{q} - \bar{q}_{kh}}_{L^2(I;\R^m)} \leq c \abs{\log k} (k + h^2)\mbox{.}
	\end{equation*}
\end{corollary}
In the case of a distributed control, the variational control discretization is associated
with an additional implementation effort.
Fully discrete strategies are therefore of independent
interest and we will investigate different variants in the following sections.

\subsubsection{Cellwise constant control approximation}	
The discrete space of controls is defined as follows
\begin{equation*}
\Qsigma(0,1) = \left\{v \in \Q(0,1) \constraintSet  v|_{I_m\times K} \in \mathcal{P}_0({I_m\times K}) \;\text{for all }K \in \mathcal{T}^\omega_h\mbox{,~} m = 1,2,\ldots,M\right\}\mbox{.}
\end{equation*}
We define the orthogonal
projection $\ProjKH \colon L^2(I\times\omega) \to \Qsigma(0,1)$ in the standard way.
Similarly, we introduce the othogonal projection $\ProjK$ onto the piecewise
constant functions in time with values in $L^2$.
Then, for any $v \in \MPRHilbert{I}{L^2}{H^1}$ there holds the projection error estimate
\begin{align}
\norm{\ProjKH v - v}_{L^2(I; L^2)} 
&\leq \norm{\ProjKH v - \ProjK v}_{L^2(I; L^2)} + \norm{\ProjK v - v}_{L^2(I; L^2)}\nonumber\\
&\leq c h \norm{\nabla v}_{L^2(I; L^2)} + c k \norm{\partial_t v}_{L^2(I; L^2)}\mbox{.}\label{eq:projection_error_piecewise_constant}
\end{align}
We obtain the following error estimate for the discretization by cellwise constant controls. Note that also in this case \cref{lemma:errorEstimatesControl_suboptimal} only provides a suboptimal estimate of order $(k + h)^{1/2}$.
\begin{theorem}[Cellwise constant controls]
	\label{thm:optimalconvergence_cellwise_constant}
	Let the assumptions of \cref{lemma:errorEstimatesControl_suboptimal} hold and suppose the piecewise 
	and cellwise constant control discretization. Then there is a constant $c > 0$ not depending on $h$ and $k$ such that
	\begin{equation*}
	%		\prodnorm{\bar{\nu} - \bar{\nu}_{kh}, \bar{q} - \bar{q}_{kh}} 
	\abs{\bar{\nu} - \bar{\nu}_{kh}} + \norm{\bar{q} - \bar{q}_{kh}}_{L^2(I\times\omega)} \leq c \abs{\log k} (k + h)\mbox{.}
	\end{equation*}
\end{theorem}
\begin{proof}
	We apply \cref{lemma:ImprovedErrorRate} with $\ProjDiscControl = \ProjKH$ and
        $q_{kh} = \ProjDiscControl\bar{q}$.
        Using the adjoint state, we write the derivative of the Lagrangian as
	\begin{equation*}
	\partial_q\mathcal{L}(\bar{\nu},\bar{q},\bar{\mu})v = \int_0^1 \bar{\nu}\inner{\alpha\bar{q}+\ControlOp^*\bar{z},v}_{L^2(\omega)}\mbox{.}
	\end{equation*}
	Abbreviating \(\bar{\lambda} = \alpha\bar{q}+\ControlOp^*\bar{z}\) and applying
        orthogonality of $\ProjKH$ and $\bar{\nu} \in \R$ we obtain
	\begin{equation*}
	\partial_{q}\mathcal{L}(\bar{\nu},\bar{q},\bar{\mu})(\ProjDiscControl\bar{q}-\bar{q})
	= \bar{\nu}\int_0^1\inner{\bar{\lambda} - \ProjDiscControl\bar{\lambda},\ProjDiscControl\bar{q}-\bar{q}}_{L^2(\omega)}\mbox{.}
	\end{equation*}
	The improved regularity $\bar{q} \in \MPRHilbert{I}{L^2(\omega)}{H^1(\omega)}$
        from \cref{prop:regularityOptimalSolution} yields
        $\norm{\ProjDiscControl\bar{q} - \bar{q}}_{L^2(I; L^2)} \leq c(k+h)$
        due to~\eqref{eq:projection_error_piecewise_constant}, and the same estimates are
        valid for \(\bar{\lambda}\) employing the same arguments.
        This results in
	\begin{equation*}
	\partial_q\mathcal{L}(\bar{\nu},\bar{q},\bar{\mu})(q_{kh}-\bar{q})
        \leq \bar{\nu}\norm{\bar{\lambda}-\ProjDiscControl\bar{\lambda}}_{\Lcontrol{2}}\norm{\ProjDiscControl\bar{q}-\bar{q}}_{\Lcontrol{2}}
	\leq c\left(k + h\right)^2\mbox{,}
	\end{equation*}
	which, combined with the estimate for \(\ProjDiscControl\bar{q} - \bar{q}\),
        yields the result.
\end{proof}

\subsubsection{Cellwise linear control approximation}
The discrete space of controls is defined as follows
\begin{align*}
Q_h &= \left\{v \in C(\overline{\omega}) \constraintSet  v|_{K} \in \mathcal P_1(K)\;\text{for all }K \in \mathcal{T}^\omega_h\right\},\\
\Qsigma(0,1) &= \left\{v \in \Q(0,1) \constraintSet  v|_{I_m\times K} \in \mathcal{P}_0(I_m; Q_h) \;\text{for all }m = 1,2,\ldots,M\right\}\mbox{.}
\end{align*}
Let $\InterpolantCellwiseLinear \colon C(\overline{\omega}) \to Q_h$ denote the Lagrange interpolant.	
We abbreviate the time indices by $\mathcal{I}_k = \set{1,2,\ldots,M}$
and decompose the set $\mathcal{I}_k\times\mathcal{T}^\omega_h$ as
\begin{align*}
\mathcal{S}_{1} &= \set{(m,K) \in \mathcal{I}_k\times\mathcal{T}^\omega_h \constraintSet  \abs{\alpha\bar{q}+\ControlOp^*\bar{z}} > 0  \text{ a.e. in } I_m \times K},\\
\mathcal{S}_{2} &= \set{(m,K) \in \mathcal{I}_k\times\mathcal{T}^\omega_h \constraintSet  \alpha\bar{q}+\ControlOp^*\bar{z} = 0 \text{ a.e. in } I_m \times K},\\
\mathcal{S}_{3} &= \left(\mathcal{I}_k\times\mathcal{T}^\omega_h\right)\setminus \left(\mathcal{S}_{1} \cup \mathcal{S}_{2}\right).
\end{align*}
Under an additional assumption we obtain the following convergence result.
\begin{theorem}[Cellwise linear controls]\label{thm:optimalconvergence_cellwise_linear}
	Adapt the assumption of \cref{lemma:errorEstimatesControl_suboptimal} and suppose the temporal piecewise constant and spatial piecewise linear control discretization.
	Assume that there is $p > d + 1$ such that $u_d \in W_0^{1,p}(\Omega)$ and that
        there is $c > 0$ such that
	\begin{equation}\label{eq:assumptionCellwiseLinearControls}
	\sum_{(m,K) \in \mathcal{S}_{3}}k_m\abs{K} \leq ch.
	\end{equation}
	Then there is a constant $c > 0$ not depending on $h$ and $k$ such that
	\begin{equation*}
	%			\norm{\bar{\nu} - \bar{\nu}_{kh}, \bar{q} - \bar{q}_{kh}} 
	\abs{\bar{\nu} - \bar{\nu}_{kh}} + \norm{\bar{q} - \bar{q}_{kh}}_{L^2(I\times\omega)}
        \leq c \abs{\log k} (k + h^{3/2-1/p})\mbox{.}
	\end{equation*}
\end{theorem}
\begin{proof}
	This proof adapts ideas from \cite[Section~5.2]{Meidner2008}.
	We set $\ProjDiscControl = \InterpolantCellwiseLinear\ProjK$ with $\ProjK$ and
        $\InterpolantCellwiseLinear$ as defined above
        to apply \cref{lemma:ImprovedErrorRate} with $q_{kh} = \ProjDiscControl\bar{q}$.
	We have to estimate the error term
        $\norm{\bar{q}-\InterpolantCellwiseLinear\ProjK\bar{q}}_{L^{2}(I\times\omega)}$.
        The temporal error is treated as before and for the spatial error
        $\ProjK(\bar{q}-\InterpolantCellwiseLinear\bar{q})$ we
        distinguish the three different cases: On
        $\mathcal{S}_1$, the local error contributions are zero
        because we either have $\bar{q} = q_a$ or $\bar{q} = q_b$,
        hence $\bar{q}-\InterpolantCellwiseLinear\ProjK\bar{q} \equiv 0$.
        On $\mathcal{S}_2$, we can use that $\bar{q}$ exhibits additional
        $L^2(I;H^2)$-regularity with an error estimate for $\InterpolantCellwiseLinear$
        that gives the rate $h^2$.        
	On $\mathcal{S}_3$, we exploit the improved regularity $\bar{q} \in C(I; W^{1,p}(\omega))$,
	see, e.g., \cite[Proposition~5.3]{Bonifacius2018},
	as well as \eqref{eq:assumptionCellwiseLinearControls} to obtain the rate $h^{3-2/p}$.
	The term $\partial_q\mathcal{L}(\bar{\nu},\bar{q},\bar{\mu})(\ProjDiscControl\bar{q}-\bar{q})$
	is estimated similarly, where we in addition use orthogonality of $\ProjK$.
	For further details on the proof, we refer to \cite[Theorem~5.21]{Bonifacius2018}.
\end{proof}
Similar assumptions to \eqref{eq:assumptionCellwiseLinearControls} have been used in
related publications for cellwise linear control discretization; see, e.g.,
\cite[Section~5.2]{Meidner2008} for a linear parabolic equation and
\cite[Theorem~4.5]{Casas2012d} for a quasilinear elliptic equation. The assumption is
justified for instance in the case that the boundary of the active set of $\bar{q}(t)$ is a
\(d-1\)-dimensional sub-manifold of \(\Omega\) at each \(t\in I\), which is often the case.

\section{Numerical examples}\label{sec:numerical_examples}

To validate the theoretical findings in practice, we consider
different numerical examples. All examples are implemented in \textsc{Matlab}. The state
constraint is incorporated into the objective functional by means of the augmented
Lagrangian method, where we employ the parameter updates suggested in
\cite[Proposition~2]{Bertsekas1976} and \cite[p.\ 404ff.]{Bertsekas1999}.
The nonlinear optimal control problems arising in each iteration of the method are then solved
using the trust-region semismooth Newton algorithm from \cite{Kunisch2015} in a monolithic
way, i.e.\ we optimize for the pair $(\nu,q)$ instead of empolying a bilevel
optimization. If the absolute value of the terminal constraint is smaller than $10^{-9}$,
the augmented Lagrangian method is stopped.

\subsection{Example with analytic reference solution}\label{sec:example1}
We consider the academic test problem
\begin{align*}
\Omega &= \omega = (0,1)^2,\quad \alpha = 1,
\quad \delta_0 = {1}/{2},\\
u_d(x) &= -2 \sin(\uppi x_1)\sin(\uppi x_2),
\quad u_0(x) = \sin(\uppi x_1)\sin(\uppi x_2),
\end{align*}
without control constraints. Moreover, we use the operator $-c\Lap$ with $c = 1/(2\pi^2)$ for convenience.
The optimal state and adjoint state are given by
\begin{equation*}
\bar{u}(t,x) = 2\left(\e^{-\bar{\nu} t} -\e^{\bar{\nu}(t-1)}\right)u_0(x),\quad
\bar{z}(t,x) = 4\e^{\bar{\nu}(t-1)}u_0(x),
\end{equation*}             
with optimal time $T = \bar{\nu} = \log(2)$.
Moreover, it can be verified that the second order sufficient optimality condition is 
satisfied on  $L^2(I\times\omega)$.
Since no control constraints are active this situation corresponds to the variational control discretization. 
We observe linear order of convergence with respect to the temporal and quadratic order of convergence with respect to the spatial discretization; see \cref{fig:example1_convergence}.

\begin{figure}
	\begin{center}
		\includegraphics{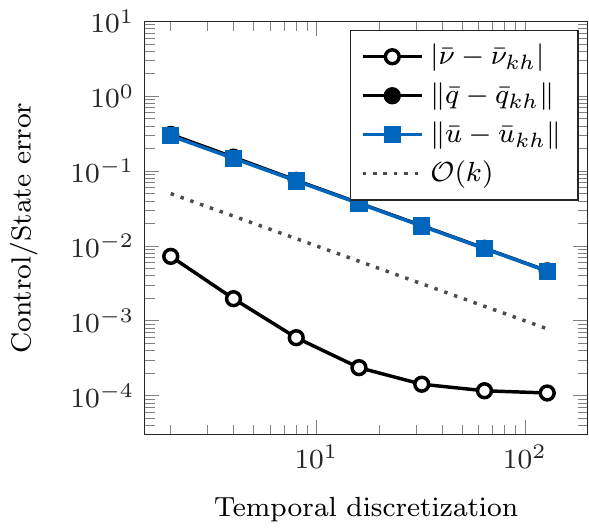}
		\includegraphics{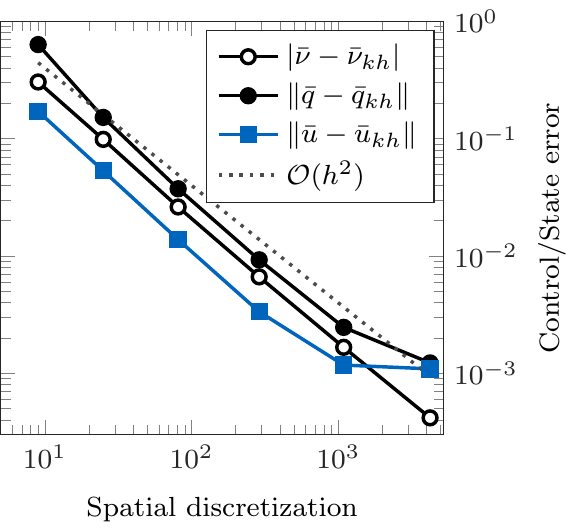}
	\end{center}
	\caption{Discretization error for Example~\ref{sec:example1} with variational control discretization and refinement of the time interval for $N = 16641$ nodes (left) and refinement of the spatial discretization for $M = 512$ time steps (right).}
	\label{fig:example1_convergence}
\end{figure}

\subsection{Example with purely time-dependent control}\label{sec:example2}
Next, we consider a time-optimal control problem with purely time-dependent controls with fixed spatially dependent functions.
Let
\begin{align*}
\Omega &= (0,1)^2\mbox{,}\quad
\omega_1 = (0, 0.5)\times(0,1),\quad \omega_2 = (0.5, 1)\times(0,0.5)\mbox{,}\quad \alpha = 10^{-2},\\
\ControlOp &\colon \R^2 \to L^2(\Omega),\quad \ControlOp q = q_1 \chi_{\omega_1} + q_2 \chi_{\omega_2},\\
\Qad(0,1) &= \{q \in L^2(I; \R^2) \constraintSet  -1.5 \leq q \leq 0\}\mbox{,}\\
u_0(x) &= 4\sin(\pi x_1^2)\sin(\pi x_2^3),\quad u_d(x) = 0, \quad \delta_0 = {1}/{10}\mbox{,}
\end{align*}
where $\chi_{\omega_1}$ and $\chi_{\omega_2}$ denote the characteristic functions on
$\omega_1$ and $\omega_2$.
The spatial mesh is chosen such that the boundaries of $\omega_1$ and $\omega_2$ coincide
with edges of the mesh, which ensures that \(B\) can be easily implemented on the discrete
level. The solutions of the discrete problem are
compared to a discrete solution calculated on a sufficiently fine mesh since no
analytic expression is available. The optimal time is $T \approx 1.79931$. We
observe linear convergence with respect to the temporal mesh size and quadratic order of
convergence with respect to the spatial mesh size, as predicted
by \cref{cor:optimalconvergence_parameter}.
\begin{figure}	
	\begin{center}
		\includegraphics{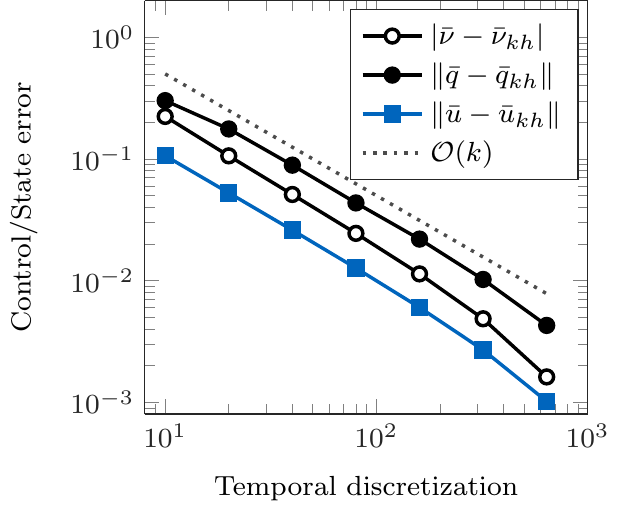}
		\includegraphics{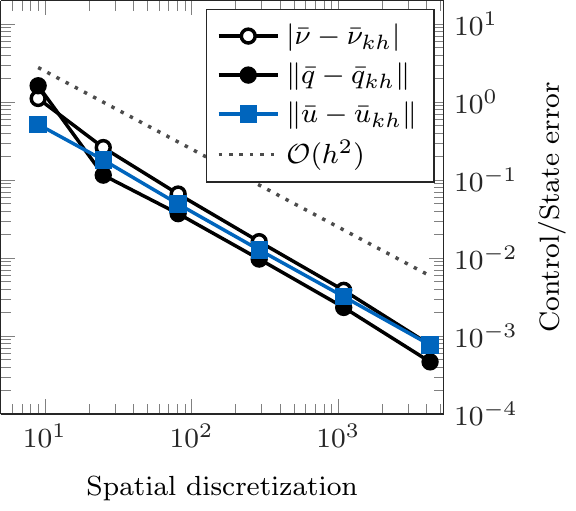}
	\end{center}
	\caption{Discretization error for Example~\ref{sec:example2} with variational control discretization and refinement of the time interval for $N = 1089$ nodes (left) and refinement of the spatial discretization for $M = 320$ time steps (right).}
	\label{fig:example2_convergence}
\end{figure}

To assess the validity of the second order sufficient optimality hypothesis, we verify
the scalar condition of \cref{lemma:ssc_reduces_to_scalar_condition} for the discrete problem. 
Since the linear system~\eqref{eq:ssc_reduces_to_scalar_condition_system} defines a symmetric but not a positive definite matrix, we calculate a solution using MINRES without assembling the matrix.
We observe that for all choices of the cost parameter $\alpha$ the condition is
satisfied on the discrete level; see \cref{table:example2_verificationssc}. Note that the
SSC for the discrete problem does not guarantee that the SSC for the continuous problem
holds. However, the fact that the numbers are robust with respect to mesh refinement can
serve as an indication for the continuous problem.
In accordance with \cref{proposition:lower_bound_coercivity}, we observe that the lower
bound of the coercivity constant~\eqref{eq:ssc_schur_scalar_condition} from
\cref{lemma:ssc_reduces_to_scalar_condition} decreases with decreasing $\alpha$.
In contrast, the constant \(\bar{\gamma}\) increases. This can be explained as follows: as
the size of the critical cone
$\CriticalCone$ decreases as $\alpha$ tends to zero and we fix $\delta\nu = 1$, the
variable $\delta\bar{q}$ has to counteract the decrease of $C_{\bar{q}}$ in order to
satisfy the linear constraint $g'(\bar{\nu}_{kh},\bar{q}_{kh})(1,\delta\bar{q}) = 0$
resulting in an increase of the norm of $\delta\bar{q}$.

\begin{table}
	\begin{center}
		\begin{footnotesize}
			\begin{tabularx}{\textwidth}{XX  rr rr rr rr}
				\toprule
				\(M\) & \(N\) & \multicolumn{2}{c}{$\alpha = 1$} & \multicolumn{2}{c}{$\alpha = 0.1$} & \multicolumn{2}{c}{$\alpha = 0.01$} & \multicolumn{2}{c}{$\alpha = 0.001$}\\ 
				\midrule
$40$ & $1089$ & $7.44$ & $4.55_{-1}$ & $17.3$ & $4.96_{-2}$ & $1.52_{+3}$ & $6.13_{-3}$ & $2.42_{+6}$ & $6.15_{-4}$\\
$80$ & $1089$ & $7.55$ & $4.56_{-1}$ & $17.7$ & $4.90_{-2}$ & $2.51_{+3}$ & $6.05_{-3}$ & $2.96_{+6}$ & $6.06_{-4}$\\ 
$160$ & $1089$ & $7.55$ & $4.53_{-1}$ & $18.1$ & $4.88_{-2}$ & $2.51_{+3}$ & $6.01_{-3}$ & $1.37_{+6}$ & $6.02_{-4}$\\ 
$320$ & $1089$ & $7.51$ & $4.51_{-1}$ & $18.3$ & $4.86_{-2}$ & $2.47_{+3}$ & $5.99_{-3}$ & $5.34_{+5}$ & $6.00_{-4}$\\ 
\midrule
$640$ & $25$ & $8.95$ & $5.53_{-1}$ & $19.2$ & $5.81_{-2}$ & $1.68_{+3}$ & $6.83_{-3}$ & $1.95_{+5}$ & $6.84_{-4}$\\
$640$ & $81$ & $7.75$ & $4.74_{-1}$ & $18.4$ & $5.08_{-2}$ & $2.21_{+3}$ & $6.18_{-3}$ & $2.18_{+5}$ & $6.19_{-4}$\\ 
$640$ & $289$ & $7.55$ & $4.54_{-1}$ & $18.3$ & $4.90_{-2}$ & $2.40_{+3}$ & $6.02_{-3}$ & $2.10_{+5}$ & $6.03_{-4}$\\ 
$640$ & $1089$ & $7.51$ & $4.49_{-1}$ & $18.2$ & $4.85_{-2}$ & $2.47_{+3}$ & $5.98_{-3}$ & $2.95_{+5}$ & $5.99_{-4}$\\ 
				\midrule
				\multicolumn{2}{p{3cm}}{Inactive constraints} & \multicolumn{2}{c}{$96\%$} & \multicolumn{2}{c}{$62\%$} & \multicolumn{2}{c}{$5\%$} & \multicolumn{2}{c}{$< 1\%$}\\
				\bottomrule
			\end{tabularx}
\vspace{.2em}
		\end{footnotesize}
	\end{center}
	\caption{Numerical verification of second order sufficient optimality condition
          for Example~\ref{sec:example2}. Table shows the quantity \(\bar{\gamma}\) of
          \cref{lemma:ssc_reduces_to_scalar_condition}, the lower bound on the coercivity
          constant for different temporal, and spatial degrees of freedoms and
          cost parameter $\alpha$.}
	\label{table:example2_verificationssc}
\end{table}

\subsection{Example with distributed control on subdomain}\label{sec:example3}
Last, we consider an example with distributed control on a subset of the domain. As before we compare to a reference solution obtained numerically on a fine grid. The problem data is
\begin{align*}
\Omega &= (0,1)^2\mbox{,} \quad
\omega = (0, 0.75)^2\mbox{,} \quad \alpha = 10^{-2},\\
\Qad(0,1) &= \{q \in L^2(I\times\omega) \constraintSet  -5 \leq q \leq 0\}\mbox{,}\\
u_d(x) &= -2\min\set{x_1, 1-x_1, x_2, 1-x_2}\mbox{,} \quad \delta_0 = {1}/{10},\\
u_0(x) &= 4\sin(\pi x_1^2) \sin(\pi x_2)^3\mbox{.}
\end{align*}
We consider the operator $-c\Lap$ with $c = 0.03$. Note, that the control acts only on a
subset $\omega \subsetneq \Omega$. Moreover, the control constraints as well as the
cost parameter are chosen in a way such that the constraints on the control are
active in a large region.

The optimal time we obtain numerically is approximately $T \approx 1.22198$. 
The control is discretized by cellwise constant functions in space. In accordance with \cref{thm:optimalconvergence_cellwise_constant} we observe linear convergence in time and space for the control variable; see \cref{fig:example3_convergence}.
In contrast, for the optimal time and the state we obtain quadratic order of convergence
in $h$, which is better than predicted by the given theory.
However, we expect that one can also prove full order of convergence for the time and
state variables and an appropriately post-processed optimal control, which is reconstructed
in terms of the adjoint state using the pointwise projection
formula~\eqref{eq:projectionFormulaControl}; see, e.g., \cite{Meidner2008,Meyer2004}.
As before, we assess the validity of the second order sufficient optimality hypothesis, by verifying
the scalar condition of \cref{lemma:ssc_reduces_to_scalar_condition} for the discrete problem. 
For all choices of the cost parameter $\alpha$, we observe that the condition is satisfied; see \cref{table:example3_verificationssc}.

\begin{figure}
	\begin{center}
		\includegraphics{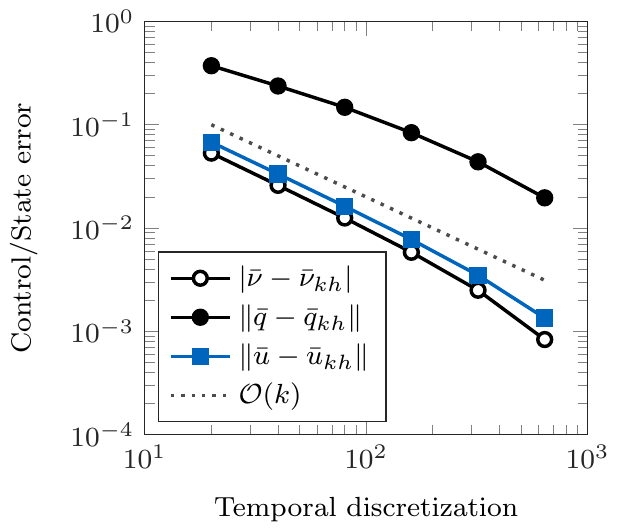}
		\includegraphics{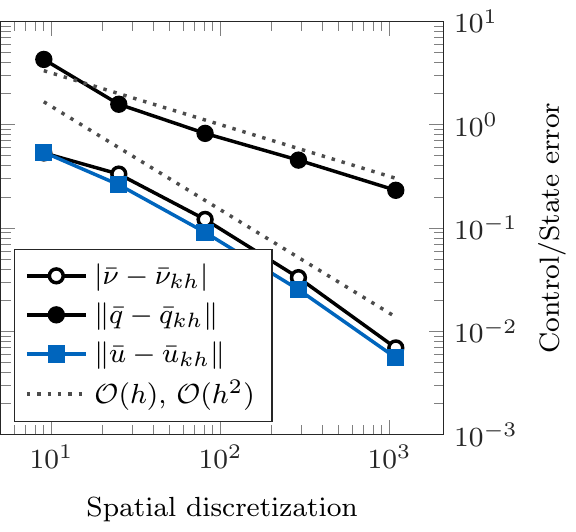}
	\end{center}
	\caption{Discretization error for Example~\ref{sec:example3} with cellwise constant control discretization and refinement of the time interval for $N = 1089$ nodes (left) and refinement of the spatial discretization for $M = 320$ time steps (right).}
	\label{fig:example3_convergence}
\end{figure}
\begin{table}
	\begin{center}
		\begin{footnotesize}				
			\begin{tabularx}{\textwidth}{XX  rr rr rr rr}
				\toprule
				$M$ & $N$ & \multicolumn{2}{c}{$\alpha = 1$} & \multicolumn{2}{c}{$\alpha = 0.1$} & \multicolumn{2}{c}{$\alpha = 0.01$} & \multicolumn{2}{c}{$\alpha = 0.001$}\\ 
				\midrule
				$20$ & $4225$ & $19.7$ & $2.46_{-1}$ & $27.5$ & $1.55_{-2}$ & $4.64_{+2}$ & $3.85_{-3}$ & $1.62_{+4}$ & $4.23_{-4}$\\ 
				$40$ & $4225$ & $17.3$ & $2.29_{-1}$ & $27.1$ & $1.51_{-2}$ & $4.57_{+2}$ & $3.77_{-3}$ & $1.52_{+4}$ & $4.13_{-4}$\\ 
				$80$ & $4225$ & $16.1$ & $2.20_{-1}$ & $26.9$ & $1.49_{-2}$ & $4.57_{+2}$ & $3.72_{-3}$ & $1.47_{+4}$ & $4.09_{-4}$\\ 
				$160$ & $4225$ & $15.6$ & $2.16_{-1}$ & $26.9$ & $1.49_{-2}$ & $4.58_{+2}$ & $3.71_{-3}$ & $1.48_{+4}$ & $4.07_{-4}$\\ 
				\midrule
				$320$ & $25$ & $2.74$ & $5.63_{-2}$ & $23.6$ & $1.20_{-2}$ & $3.31_{+2}$ & $2.68_{-3}$ & $4.00_{+4}$ & $2.96_{-4}$\\ 
				$320$ & $81$ & $9.97$ & $1.52_{-1}$ & $26.5$ & $1.34_{-2}$ & $3.81_{+2}$ & $3.27_{-3}$ & $1.75_{+4}$ & $3.65_{-4}$\\ 
				$320$ & $289$ & $13.7$ & $1.97_{-1}$ & $26.5$ & $1.44_{-2}$ & $4.46_{+2}$ & $3.59_{-3}$ & $1.58_{+4}$ & $3.95_{-4}$\\ 
				$320$ & $1089$ & $14.9$ & $2.10_{-1}$ & $26.8$ & $1.47_{-2}$ & $4.60_{+2}$ & $3.68_{-3}$ & $1.39_{+4}$ & $4.03_{-4}$\\ 
				\midrule
				\multicolumn{2}{p{3cm}}{Inactive constraints} & \multicolumn{2}{c}{$98\%$} & \multicolumn{2}{c}{$67\%$} & \multicolumn{2}{c}{$19\%$} & \multicolumn{2}{c}{$6\%$}\\
				\bottomrule
			\end{tabularx}
\vspace{.2em}
		\end{footnotesize}
	\end{center}
	\caption{Numerical verification of second order sufficient optimality condition
          for Example~\ref{sec:example3}. Table shows the quantity \(\bar{\gamma}\) of
          \cref{lemma:ssc_reduces_to_scalar_condition}, the lower bound on the coercivity
          constant for different temporal and spatial degrees of freedom, and the cost
          parameter $\alpha$.}
	\label{table:example3_verificationssc}
\end{table}

\appendix

\section{Stability estimates}
\begin{proposition}\label{prop:stabilityS}
	There exists a constant $c > 0$ such that for all
	$\nu > 0$, $q \in L^2(I; H^{-1})$, and initial conditions $u_0 \in L^2$ it holds
	\begin{align*}
	\norm{u}_{C([0,1]; L^2)} +  \sqrt{\nu}\norm{u}_{L^2(I; H^1)} &\leq c\left(\sqrt{\nu} \norm{q}_{L^2(I;H^{-1})} + \norm{u_0}_{L^2}\right)\mbox{,}\displaybreak[1]\\
	\norm{\delta u}_{C([0,1]; L^2)} + \sqrt{\nu}\norm{\delta u}_{L^2(I; H^1)} &\leq c\frac{\abs{\delta\nu}}{\sqrt{\nu}}\left(\norm{q}_{L^2(I;H^{-1})} + \norm{u}_{L^2(I;H^{1})}\right)\\
	&\quad + \sqrt{\nu}\norm{\delta q}_{L^2(I;H^{-1})}\nonumber\mbox{,}\displaybreak[1]\\
	\norm{\delta\tilde{u}}_{C([0,1]; L^2)} +  \sqrt{\nu}\norm{\delta\tilde{u}}_{L^2(I; H^1)} &\leq c\frac{\abs{\delta\nu}}{\sqrt{\nu}}\left(\norm{\delta q}_{L^2(I; H^{-1})} + \norm{\delta u}_{L^2(I; H^{1})}\right)\mbox{,}
	\end{align*}
	where $u = S(\nu,q)$, $\delta u = S'(\nu,q)(\delta\nu, \delta q)$ and $\delta\tilde{u} = S''(\nu,q)[\delta\nu, \delta q]^2$
	for $\delta\nu \in \R$ and $\delta q \in L^2(I; H^{-1})$.
	Furthermore, for $q_1, q_2 \in \Qad(0,1)$ we have
	\begin{gather*}
	\norm{u_1-u_2}_{C([0,1]; L^2)} + \sqrt{\nu_1}\norm{u_1-u_2}_{L^2(I; H^1)}\leq 
	c\left(\abs{\nu_1-\nu_2}
	+ \norm{q_1-q_2}_{L^2(I;H^{-1})}\right)\mbox{,}\\
	\norm{\delta u_1 - \delta u_2}_{C([0,1]; L^2)} \leq c\left(\abs{\nu_1-\nu_2} + \norm{q_1-q_2}_{L^2(I;H^{-1})}\right)\left(\abs{\delta\nu} + \norm{\delta q}_{L^2(I;H^{-1})}\right)\mbox{,}
	\end{gather*}
	where $u_i = S(\nu_i,q_i)$ and $\delta u_i = S'(\nu_i,q_i)(\delta\nu, \delta q)$ for $i = 1,2$ and
	\begin{align*}
	c_0 = c_0(\nu_1,\nu_2) &= c/\sqrt{\nu_1} \max \set{1, 1/\sqrt{\nu_2}, \nu_2},\\
	c_1 = c_1(\nu_1,\nu_2) &= c/\sqrt{\nu_1} \max \set{1, 1/\nu_1, 1/(\nu_1\sqrt{\nu_2}), 1/\nu_2, 1/\nu_2^{3/2}, \nu_2/\nu_1}.
	\end{align*}
	The constant $c > 0$ depends exclusively on Poincar\'e's constant, $\Qad$, and $u_0$.
\end{proposition}
\begin{proof}
	This follows from standard energy estimates, see, e.g., \cite[\S26]{Wloka1987},
	and the equations for state, linearized state and second linearized state; see \cref{lemma:control_to_state_differentiable}.
\end{proof}

\section{Discretization error estimates}
\label{appendix:discretization_error_estimates}
In this section we collect
error estimates for space-time finite element discretizations.

\begin{lemma}\label{lemma:errorEstimatesStateEquation}
	Let $\nu \in \Rplus$ and $f \in L^2((0,1); L^2)$. For the solution $u = u(\nu,f)$ to the state equation with right-hand side $f$ and the discrete solution $u_{kh} = u_{kh}(\nu,f)$ to equation~\eqref{eq:stateEquationDiscrete} with right-hand side $f$ it holds
	\begin{align}
	\norm{u - u_{kh}}_{L^2(I; L^2)} &\leq c\left(k\norm{\partial_t u}_{L^2(I; L^2)} + h^2 \norm{\Lap u}_{L^2(I; L^2)}\right)\label{eq:errorEstimatesStateEquationL2L2}\mbox{,}\\
	\norm{\nabla u - \nabla u_{kh}}_{L^2(I; L^2)} &\leq c(k^{1/2}+h)\left(\norm{\partial_t u}_{L^2(I; L^2)} + \norm{\Lap u}_{L^2(I; L^2)}\right)\label{eq:errorEstimatesStateEquationL2H1}\mbox{,}
	\end{align}
	where the constant $c$ is independent of $\nu, k, h, f, u_0$, and $u$.
\end{lemma}
\begin{proof}
	The estimates are shown as in \cite{Meidner2008a}, where we clearly see that the constants are independent of $\nu$.
\end{proof}

\begin{lemma}\label{lemma:errorEstimatesStateEquationLinf}
	Let $\nu \in \Rplus$ and $f \in L^\infty((0,1); L^2)$. For the solution $u = u(\nu,f)$ to the state equation with right-hand side $f$ and the discrete solution $u_{kh} = u_{kh}(\nu,f)$ to equation~\eqref{eq:stateEquationDiscrete} with right-hand side $f$ it holds
	\begin{align}	
	\norm{u(1) - u_{kh}(1)}_{L^2} &\leq c\abs{\log k}\left(k + h^2\right)\left((1 + \nu)\norm{f}_{L^\infty(I; L^2)} + \nu^{-1}\norm{u_0}_{L^2}\right)\label{eq:errorEstimatesStateEquationLinfL2}\mbox{,}\\
	\norm{u(1) - u_{kh}(1)}_{L^2} &\leq c\abs{\log k}\left(k + h^2\right)(1 + \nu)\left(\norm{f}_{L^\infty(I; L^2)} + \norm{\Lap u_0}_{L^2}\right)\label{eq:errorEstimatesStateEquationLinfL2_stable}\mbox{,}
	\end{align}
	where the constant $c$ is independent of $\nu, k, h, f, u_0$, and $u$.
\end{lemma}
\begin{proof}
	The error estimate~\eqref{eq:errorEstimatesStateEquationLinfL2_stable} is shown in \cite[Section~5]{Meidner2011a}, 
	where the dependence on $\nu$ can be traced from the proof.	
	For the estimate~\eqref{eq:errorEstimatesStateEquationLinfL2}, 
	consider first the case $u_0 = 0$. This is exactly~\eqref{eq:errorEstimatesStateEquationLinfL2_stable}.
	In the case $q = 0$, we combine Theorems~1 and~2 from \cite{Luskin1982} with clearly stated time dependency. Superposition of both estimates yields~\eqref{eq:errorEstimatesStateEquationLinfL2}.	
\end{proof}

\newpage
\bibliographystyle{siamplain}

\end{document}